\documentclass[a4paper]{amsart}
\usepackage{txfonts, amsmath,amstext,amsthm,amscd,amsopn,verbatim,amssymb, amsfonts}
\usepackage{fullpage}

\usepackage{float}
\restylefloat{table}

\usepackage{tikz}
\usepackage{tikz-cd}

\usepackage[all,cmtip]{xy}

\usetikzlibrary{matrix}
\usetikzlibrary{shapes}
\usetikzlibrary{arrows}
\usetikzlibrary{calc,3d}
\usetikzlibrary{decorations,decorations.pathmorphing,decorations.pathreplacing,decorations.markings}
\usetikzlibrary{through}

\tikzset{ext/.style={circle, draw,inner sep=1pt},int/.style={circle,draw,fill,inner sep=1.4pt},nil/.style={inner sep=1pt}}
\tikzset{cy/.style={circle,draw,fill,inner sep=2pt},scy/.style={circle,draw,inner sep=2pt},scyx/.style={draw,cross out,inner sep=2pt},scyt/.style={draw,regular polygon,regular polygon sides=3,inner sep=0.95pt}}
\tikzset{exte/.style={circle, draw,inner sep=3pt},inte/.style={circle,draw,fill,inner sep=3pt}}
\tikzset{diagram/.style={matrix of math nodes, row sep=3em, column sep=2.5em, text height=1.5ex, text depth=0.25ex}}
\tikzset{diagram2/.style={matrix of math nodes, row sep=0.5em, column sep=0.5em, text height=1.5ex, text depth=0.25ex}}

\tikzset{
  rightblue/.style={
    decoration={markings,mark=at position .8 with {\arrow[scale=1.2,blue]{latex}}},
    postaction={decorate},begin
    shorten >=0.4pt}}
\tikzset{
  leftblue/.style={
    decoration={markings,mark=at position .6 with {\arrowreversed[scale=1.2,blue]{latex}}},
    postaction={decorate},
    shorten >=0.4pt}}
\tikzset{
  rightred/.style={
    decoration={markings,mark=at position .4 with {\arrow[scale=1.2,red]{latex}}},
    postaction={decorate},
    shorten >=0.4pt}}
\tikzset{
  leftred/.style={
    decoration={markings,mark=at position .2 with {\arrowreversed[scale=1.2,red]{latex}}},
    postaction={decorate},
    shorten >=0.4pt}}
    
\tikzset{
  crossed/.style={
    decoration={markings,mark=at position .5 with {\arrow{|}}},
    postaction={decorate},
    shorten >=0.4pt}}

\newcommand{\Ed}{{
\begin{tikzpicture}[baseline=-.8ex,scale=.5]
\node[nil] (a) at (0,0) {};
\node[nil] (b) at (1,0) {};
\draw (a) edge[-latex] (b);
\end{tikzpicture}}}

\newcommand{\dE}{{
\begin{tikzpicture}[baseline=-.8ex,scale=.5]
\node[nil] (a) at (0,0) {};
\node[nil] (b) at (1,0) {};
\draw (a) edge[latex-] (b);
\end{tikzpicture}}}

\newcommand{\EdE}{{
\begin{tikzpicture}[baseline=-.65ex,scale=.5]
 \node[nil] (a) at (0,0) {};
 \node[int] (b) at (1,0) {};
 \node[nil] (c) at (2,0) {};
 \draw (a) edge[-latex] (b);
 \draw (b) edge[latex-] (c);
\end{tikzpicture}}}

\newcommand{\Ess}{{
\begin{tikzpicture}[baseline=-.65ex,scale=.5]
 \node[nil] (a) at (0,0) {};
 \node[nil] (c) at (1.4,0) {};
 \draw (a) edge[crossed,->] (c);
\end{tikzpicture}}}

\newcommand{\ET}{{
\begin{tikzpicture}[baseline=-.65ex,scale=.5]
 \node[nil] (a) at (0,0) {};
 \node[nil] (c) at (1,0) {};
 \draw (a) edge[very thick,->] (c);
\end{tikzpicture}}}

\usepackage{chngcntr}
\counterwithin{figure}{section}
\theoremstyle{plain}
  \newtheorem{thm}[figure]{Theorem}
  \newtheorem{defi}[figure]{Definition}
  \newtheorem{prop}[figure]{Proposition}
  \newtheorem{defprop}[figure]{Definition/Proposition}
  \newtheorem{cor}[figure]{Corollary}
  
  \newtheorem{lemma}[figure]{Lemma}
\theoremstyle{definition}
  
  \newtheorem{rem}[figure]{Remark}

\newcommand{\K}{{\mathbb{K}}}
\newcommand{\Z}{{\mathbb{Z}}}

\newcommand{\sym}{{\mathbb{S}}}

\newcommand{\GC}{\mathrm{GC}}
\newcommand{\HGC}{\mathrm{HGC}}

\newcommand{\RGC}{\mathrm{RGC}}
\newcommand{\OGC}{\mathrm{OGC}}
\newcommand{\SGC}{\mathrm{SGC}}

\newcommand{\G}{\mathrm{G}}
\newcommand{\HG}{\mathrm{HG}}
\newcommand{\OG}{\mathrm{OG}}

\newcommand{\mV}{\mathrm{V}}
\newcommand{\mE}{\mathrm{E}}
\newcommand{\mH}{\mathrm{H}}
\newcommand{\mB}{\mathrm{B}}
\newcommand{\mD}{\mathrm{D}}
\newcommand{\mO}{\mathrm{O}}
\newcommand{\mS}{\mathrm{S}}

\DeclareMathOperator{\sgn}{sgn}
\DeclareMathOperator{\id}{id}
\DeclareMathOperator{\Def}{Def}

\newcommand{\grac}{\mathrm{grac}}
\newcommand{\Grac}{\mathrm{Grac}}
\newcommand{\rgra}{\mathrm{rgra}}
\newcommand{\RGra}{\mathrm{RGra}}

\newcommand{\bu}{{\bullet}}

\newcommand{\LieB}{\textsf{LieB}}

\newcommand{\hoLieB}{\textsf{hoLieB}}

\begin{document}
\title{Hairy graphs to ribbon graphs via a fixed source graph complex}

\author{Assar Andersson}
\author{Marko \v Zivkovi\' c}
\address{Mathematics Research Unit, University of Luxembourg\\ 
Maison du Nombre, 6, avenue de la Fonte, L-4364 Esch-sur-Alzette\\
Grand Duchy of Luxembourg}



\begin{abstract}
We show that the hairy graph complex $(\HGC_{n,n},d)$ appears as an associated graded complex of the oriented graph complex $(\mO \GC_{n+1},d)$, subject to the filtration on the number of targets, or equivalently sources, called the fixed source graph complex. The fixed source graph complex $(\mO \GC_1,d_0)$ maps into the ribbon graph complex $\RGC$, which models the moduli space of Riemann surfaces with marked points. The full differential $d$ on the oriented graph complex $\mO \GC_{n+1}$ corresponds to the deformed differential $d+h$ on the hairy graph complex $\HGC_{n,n}$, where $h$ adds a hair. This deformed complex $(\HGC_{n,n},d+h)$ is already known to be quasi-isomorphic to standard Kontsevich's graph complex $\GC^2_n$. This gives a new connection between the standard and the oriented version of Kontsevich's graph complex.
\end{abstract}

\maketitle

\section{Introduction}

The main motivation for the present work is the explicit morphism of complexes
$$F:(\mO  \GC_1, \delta) \to (\RGC[1], \delta+\Delta_1)$$
constructed by S.\ Merkulov and T.\ Willwacher in \cite{MW}. 
Here $(\mO  \GC_1, \delta)$ is the oriented version of Kontsevich's graph complex, and $ (\RGC, \delta+\Delta_1)$ is a complex of ribbon graphs.
Ribbon graphs (sometimes called fat graphs) models Riemann surfaces with marked points \cite{Penner}. For our ribbon graph complex $\RGC$, we get 
\begin{equation}
H^k(\RGC,\delta) \cong \prod_{g,n} \left(H_c^{k-n}(\mathcal{M}_{g,n},\mathbb{Q}) \otimes\sgn_n\right)^{\sym_n} \oplus \begin{cases} \mathbb{Q} & \text{for } k=1,5,9,\ldots\\
0 &\text{otherwise,} \end{cases}
\end{equation}
where $H_c(\mathcal{M}_{g,n},\mathbb{Q})$ is the compact support cohomology of the moduli space of Riemann surfaces of genus $g$ with $n$ marked points, see \cite{Kont3} for more details. 
In this context, the differential $\delta+\Delta_1$ constructed in \cite{MW} is a deformation of the classical differential $\delta$ on $\RGC$. A simple observation gives that the same explicit formula $F$ is also a map of complexes
\begin{equation}\label{eq:Fintro}
F:(\mO  \GC_1, \delta_0) \to (\RGC[1], \delta),
\end{equation}
where it is instead the differential on $\mO\GC_1$ that is not standard. The differential $\delta_0$ splits vertices of graphs in $\mO \GC$ in a way that preserves the number of target vertices.  To the authors best knowledge, the \emph{(oriented) fixed target graph complex} $(\mO  \GC_n, \delta_0)$ has not been studied earlier. In this paper, we show that there is a quasi-isomorphism from the oriented graph complex with this new differential to the better known hairy graph complex $\HGC_{n,n}$, studied in e.g.\ \cite{AT}, \cite{DGC2}. 
\begin{thm}\label{thm:main}
There is a map of graded vector spaces
$$G:\mO  \GC_{n+1} \to \HGC_{n,n}$$
such that the associated morphisms of complexes
$$G:(\mO  \GC_{n+1}, \delta_0)  \to (\HGC_{n,n},\delta).$$
and
$$G:(\mO \GC_{n+1},\delta) \to (\HGC_{n,n},\delta+\chi)$$
are quasi-isomorphisms. Here $\chi$ is the extra differential on $\HGC_{n,n}$ that adds a hair, considered in \cite{DGC2}. 
\end{thm} 
As a corollary, we get a relationship between the ribbon graph complex and the hairy graph complex.
\begin{cor} \label{cor:Mgn}
We have an explicit zig-zag of morphisms
$$(\HGC_{0}, \delta)\gets (\mO \GC_1, \delta_0)\to (\RGC[1],\delta),$$
where the left map is a quasi-isomorphism.
\end{cor} 
A recent result by  M. Chan, S. Galatius and S. Payne \cite{CGP2} states that there exists an embedding 
$$H^k(\GC^{marked}_0, \delta)\to \prod_{g,n} H^{k-n+1}_c(\mathcal{M}_{g,n},\mathbb{Q}).$$
Here, $\GC^{marked}$ is a complex of hairy graphs where each hair is labeled by an integer. However, no explicit map is given. After symmetrizing  both sides and using Theorem \ref{thm:main}, we get that there exists an embedding 
$$H(\mO \GC_1,\delta_0) \to \prod_{g,n} \left( H_c(\mathcal{M}_{g,n},\mathbb{Q}) \otimes \sgn_n\right)^{\sym_n}\oplus \begin{cases} \mathbb{Q} & \text{for } k=1,5,9,\ldots\\
0 &\text{otherwise,}\end{cases}.$$
We conjecture that this embedding is given explicitly by the map $H(F)$. This conjecture is also given in \cite{MW}, and to support it, it is shown that $H(F)$ is nontrivial on all loop classes of $\mO \GC_1.$ 
 
Let $(\mS \G_n,d)$ be the graph complex consisting of directed graphs that contain at least one source vertex, with the edge contracting differential $d$.
 In \cite{MultiSourced}, the second author showed that the projection
$$
\left(\mS \G_n,d \right)\to \left(\mO \G_n,d\right)
$$
is a quasi-isomorphism.
\begin{rem}
The graph complex $(\mS \GC_n,\delta)$ with the vertex splitting differential $\delta$, is dual of the complex $(\mS \G_n,d)$ with the edge contraction differential $d$. The space $\mS \G_n$ is spanned by formal linear combinations of graphs, while its dual space $\mS \GC_n$ is spanned by formal series of the same graphs.

Each result regarding a graph complex $(\G,d)$ transfers to a dual result regarding $(\GC,\delta)$. e.g.\ the map $G:(\mO \GC_{n+1}, \delta_0) \to (\HGC_{n,n},\delta)$ has a dual map $\Phi:(\HG_{n,n},d) \to (\OG_{n+1},{d_0})$, which is also a quasi-isomorphism.
See Subsection \ref{ss:dual} for more details.
\end{rem}

In this paper, we show that this relation between oriented graphs and sourced graphs also holds true when we consider the \emph{fixed source graph complexes}, with differentials $d_0$ that preserve the number of source vertices. 

\begin{prop} \label{prop:sourced}
The projection
$$
(\mS \GC_n,\delta_0)\to(\mO \GC_n,\delta_0),
$$
is a quasi-isomorphism.
\end{prop}

\begin{rem}
In \eqref{eq:Fintro}, we consider the differential $\delta_0$ that preserves the number of target vertices as opposed to source vertices. However, as there is an isomorphism $Inv:\mO\GC\to \mO\GC$ that inverts all the edges, thus mapping source vertices to target vertices and vice versa, we may shift freely between considering source vertices and target vertices. 
To keep notation consistent with what is previously established in \cite{MultiSourced}, we shall mainly consider the differential that preserves the number of source vertices, which we will also denote by $\delta_0$.
\end{rem}

Let $\GC_n^2$  $(\HGC_{n,n}^2)$ denote the version of Kontsevich's standard (hairy) graph complex that includes graphs with $2$-valent vertices.
The following result connects Kontsevich's graph complex and hairy graph complex.
\begin{thm}[\cite{TW1},\cite{TW2},\cite{Will},\cite{DGC2}]
There is a quasi-isomorphism
$$
\left(\GC_n^2,\delta\right)\rightarrow\left(\HGC^2_{n,n},\delta+\chi\right),
$$
by summing over all ways to attach a hair to each graph.
\end{thm}
Together with Theorem \ref{thm:main}, we arrive to the following corollary.
\begin{cor} \label{cor:Mgn}
There is a zig-zag of explicit quasi-isomorphisms
$$
\left(\GC_n^2,\delta\right)
\rightarrow\left(\HGC^2_{n,n},\delta+\chi\right)
\hookleftarrow\left(\HGC_{n,n},\delta+\chi\right)
\leftarrow\left(\OGC_{n+1},\delta\right)
\hookrightarrow(\SGC_{n+1},\delta).
$$
\end{cor}

The previous corollary gives a fourth proof that the cohomology of Kontsevich's graph complex is equal to the cohomology of the oriented graph complex, together with  \cite{MW}, \cite{oriented}, and \cite{Multi}. The proof from \cite{Multi} gave an explicit map
$$
\left(\G_n,d \right)
\rightarrow\left(\OG_{n+1}^\varnothing, d\right),
$$
where the superscript $\varnothing$, means that the loop graphs are omitted. This proof gives explicit maps too, and this time they are also naturally defined for loops.

\subsection*{Structure of the paper}

In Section \ref{s:def} we recall the required definitions and some results. Section \ref{s:map} introduces the map between the oriented and the hairy graph complexes, showing that it is indeed a map of complexes. Section \ref{s:qi} contains our main results about quasi-isomorphisms. In Section \ref{s:rib} we recall the definition of Ribbon graphs and the map $F$ of Merkulov and Willwacher from \cite{MW}, giving sense to our motivation.

\subsection*{Acknowledgements}
We are grateful to Sergei Merkulov and Alexey Kalugin for many useful discussions and comments.
\section{Required definitions and results}
\label{s:def}

In this section we recall basic notation and several results shown in the literature that will be used in the paper.

\subsection{General notation}

We work over a field $\K$ of characteristic zero. All vector spaces and differential graded vector spaces are assumed to be $\K$-vector spaces. 

Graph complexes as vector spaces are generally defined by the graphs that span them. When we say \emph{a single term graph} in a graph complex, we mean the base graph, while any linear combination (or a series) of graphs will be called \emph{an element} of the graph complex.


\subsection{Directed, oriented and sourced graph complexes}
The directed, oriented and sourced graph complexes $\mD\GC_{n}$, $\mO\GC_{n}$ and $\mS\GC_n$ are defined in \cite{MultiSourced}. In this paper, we will only consider single directional complexes. Accordingly, let us recall a simplified definition.

Consider the set of directed connected graphs $\bar\mV_v\bar\mE_e\grac^{2}$ with $v>0$ distinguishable vertices and $e\geq 0$ distinguishable directed edges, all vertices being at least $2$-valent, and without tadpoles (edges that start and end at the same vertex). We also ban passing vertices, i.e.\ 2-valent vertices with one incoming and one outgoing edge. This condition will not change the homology, as shown in \cite[Subsection 3.2]{MultiSourced}.

Let $\bar\mV_v\bar\mE_e\mO\grac^{2}\subset\bar\mV_v\bar\mE_e\grac^{2}$ be the subset of all graphs without closed paths along the directed edges. For $s\geq 0$, let $\bar\mV_v\bar\mE_e\mS_s\grac^{2}\subset\bar\mV_v\bar\mE_e\grac^{2}$ be the subset of graphs that have exactly $s$ sources, i.e.\ vertices without incoming edges.

For $n\in\Z$, let the degree of an element of $\bar\mV_v\bar\mE_e\grac^{\geq 2}$ be $d=n-vn-(1-n)e$. Let
\begin{equation}
\bar\mV_v\bar\mE_e\mD\G_n:=\langle\bar\mV_v\bar\mE_e\grac^{2}\rangle[n-vn-(1-n)e],
\end{equation}
\begin{equation}
\bar\mV_v\bar\mE_e\mO\G_n:=\langle\bar\mV_v\bar\mE_e\mO\grac^{2}\rangle[n-vn-(1-n)e],
\end{equation}
\begin{equation}
\bar\mV_v\bar\mE_e\mS_s\G_n:=\langle\bar\mV_v\bar\mE_e\mS_s\grac^{2}\rangle[n-vn-(1-n)e],
\end{equation}
be the vector spaces of formal linear combinations of some elements of $\bar\mV_v\bar\mE_e\grac^{2}$ with coefficients in $\K$. They are graded vector spaces with non-zero terms only in degree $d=n-vn-(1-n)e$.

There are natural right actions of the group $\sym_v\times\sym_e$ on $\bar\mV_v\bar\mE_e\grac^{2}$, $\bar\mV_v\bar\mE_e\mO\grac^{2}$ and $\bar\mV_v\bar\mE_e\mS_s\grac^{2}$, where $S_v$ permutes vertices and $S_e$ permutes edges.
Let $\sgn_v$ and $\sgn_e$ be one-dimensional representations of $S_v$, respectively $S_e$, where the odd permutation reverses the sign. They can be considered as representations of the product $\sym_v\times\sym_e$. Let us consider the spaces of invariants:

\begin{equation}
\mV_v\mE_e\mD\G_{n}:=\left\{
\begin{array}{ll}
\left(\bar\mV_v\bar\mE_e\mD\G_n\otimes\sgn_e\right)^{\sym_v\times\sym_e}
\qquad&\text{for $n$ even,}\\
\left(\bar\mV_v\bar\mE_e\mD\G_n\otimes\sgn_v\right)^{\sym_v\times\sym_e}
\qquad&\text{for $n$ odd,}
\end{array}
\right.
\end{equation}
\begin{equation}
\mV_v\mE_e\mO\G_{n}:=\left\{
\begin{array}{ll}
\left(\bar\mV_v\bar\mE_e\mO\G_n\otimes\sgn_e\right)^{\sym_v\times\sym_e}
\qquad&\text{for $n$ even,}\\
\left(\bar\mV_v\bar\mE_e\mO\G_n\otimes\sgn_v\right)^{\sym_v\times\sym_e}
\qquad&\text{for $n$ odd,}
\end{array}
\right.
\end{equation}
\begin{equation}
\mV_v\mE_e\mS_s\G_{n}:=\left\{
\begin{array}{ll}
\left(\bar\mV_v\bar\mE_e\mS_s\G_n\otimes\sgn_e\right)^{\sym_v\times\sym_e}
\qquad&\text{for $n$ even,}\\
\left(\bar\mV_v\bar\mE_e\mS_s\G_n\otimes\sgn_v\right)^{\sym_v\times\sym_e}
\qquad&\text{for $n$ odd.}
\end{array}
\right.
\end{equation}

As the group is finite, the space of invariants may be replaced by the space of coinvariants. The underlying vector space of 
the \emph{directed graph complex} is given by
\begin{equation}
\mD\G_{n}:=\bigoplus_{v\geq 1,e\geq 0}\mV_v\mE_e\mD\G_{n}.
\end{equation}
The underlying vector space of 
the \emph{oriented graph complex} is given by
\begin{equation}
\mO\G_{n}:=\bigoplus_{v\geq 1,e\geq 0}\mV_v\mE_e\mO\G_{n}.
\end{equation}
The underlying vector space of 
the \emph{sourced graph complex} with $s$ sources is given by
\begin{equation}
\mS_s\G_{n}:=\bigoplus_{v\geq 1,e\geq 0}\mV_v\mE_e\mS_s\G_{n},
\end{equation}
and the full sourced graph complex is given by
\begin{equation}
\mS\G_n:=\bigoplus_{s\geq 1}\mS_s\G_{n}.
\end{equation}

Dual spaces of those defined above are spanned by the same graphs, but infinite sums are allowed. Since $\mV_v\mE_e\mD\G_{n}$ are finitely dimensional, this does not make difference. But duals of total complexes are defined as:
\begin{equation}
\mD\GC_{n}:=\prod_{v\geq 1,e\geq 0}\mV_v\mE_e\mD\G_{n},
\end{equation}
\begin{equation}
\mO\GC_{n}:=\prod_{v\geq 1,e\geq 0}\mV_v\mE_e\mO\G_{n},
\end{equation}
\begin{equation}
\mS_s\GC_{n}:=\prod_{v\geq 1,e\geq 0}\mV_v\mE_e\mS_s\G_{n},
\end{equation}
\begin{equation}
\mS\GC_n:=\prod_{s\geq 1}\mS_s\G_{n}.
\end{equation}

Note that every oriented graph is sourced and both of them are directed, so there are inclusions and a projections
\begin{equation}
\mO\GC_{n}\hookrightarrow\mS\GC_{n}\hookrightarrow\mD\GC_{n}, \quad
\mD\G_{n}\twoheadrightarrow\mS\G_{n}\twoheadrightarrow\mO\G_{n}.
\end{equation}

\subsection{Hairy graph complex}
The hairy graph complexe $\HGC_{m,n}$ is in general defined e.g.\ in \cite{DGC2}.
In this paper we are interested in the case when $m=n$, that is $\HGC_{m,n}$. For simplicity, we will use the shorter notation $\HGC_n:=\HGC_{n,n}$.

Let us quickly recall the definition, similar to the one of oriented and sourced graph complexes.
Consider the set of directed connected graphs $\bar\mV_v\bar\mE_e\bar\mH_s\grac^{3}$ with $v>0$ distinguishable vertices, $e\geq 0$ distinguishable directed edges and $s\geq 0$ distinguishable hairs attached to some vertices, all vertices being at least $3$-valent, and without tadpoles (edges that start and end at the same vertex). In hairy graphs, the valence also includes the attached hairs, i.e.\ the valence of a vertex is the number of edges and hairs attached to it.

For $n\in\Z$, let the degree of an element of $\bar\mV_v\bar\mE_e\bar\mH_s\grac^{3}$ be $d=n-vn-(1-n)e-s$. Let
\begin{equation}
\bar\mV_v\bar\mE_e\bar\mH_s\G_n:=\langle\bar\mV_v\bar\mE_e\bar\mH_s\grac^{3}\rangle[n-vn-(1-n)e-s]
\end{equation}
be the vector space of formal series of $\bar\mV_v\bar\mE_e\bar\mH_s\grac^{3}$ with coefficients in $\K$. It is a graded vector space with a non-zero term only in degree $d=n-vn-(1-n)e-s$.

There is a natural right action of the group $\sym_v\times \sym_s\times \left(\sym_e\ltimes \sym_2^{\times e}\right)$ on $\bar\mV_v\bar\mE_e\bar\mH_s\grac^{3}$, where $S_v$ permutes vertices, $S_s$ permutes hairs, $S_e$ permutes edges and $S_2^{\times e}$ changes the direction of edges.
Let $\sgn_v$, $\sgn_s$, $\sgn_e$ and $\sgn_2$ be one-dimensional representations of $S_v$, respectively $S_s$, respectively $S_e$, respectively $S_2$, where the odd permutation reverses the sign. They can be considered as representations of the whole product $\sym_v\times \sym_s\times \left(\sym_e\ltimes \sym_2^{\times e}\right)$. Let us consider the space of invariants:
\begin{equation}
\mV_v\mE_e\mH_s\G_{n}:=\left\{
\begin{array}{ll}
\left(\bar\mV_v\bar\mE_e\bar\mH_s\G_n\otimes\sgn_e\otimes\sgn_s\right)^{\sym_v\times \sym_s\times \left(\sym_e\ltimes \sym_2^{\times e}\right)}
\qquad&\text{for $n$ even,}\\
\left(\bar\mV_v\bar\mE_e\bar\mH_s\G_n\otimes\sgn_v\otimes\sgn_s\otimes\sgn_2^{\otimes e}\right)^{\sym_v\times \sym_s\times \left(\sym_e\ltimes \sym_2^{\times e}\right)}
\qquad&\text{for $n$ odd.}
\end{array}
\right.
\end{equation}

Again, because the group is finite, the space of invariants may be replaced by the space of coinvariants.
The \emph{hairy graph complex} with $s$ hairs is
\begin{equation}
\mH_s\G_{n}:=\bigoplus_{v\geq 1,e\geq 0}\mV_v\mE_e\mH_s\G_{n},
\end{equation}
and the general one is
\begin{equation}
\mH\G_n:=\bigoplus_{s\geq 1}\mH_s\G_{n}.
\end{equation}
Duals are:
\begin{equation}
\mH_s\GC_{n}:=\prod_{v\geq 1,e\geq 0}\mV_v\mE_e\mH_s\G_{n},
\end{equation}
\begin{equation}
\mH\GC_n:=\prod_{s\geq 1}\mH_s\G_{n}.
\end{equation}

\subsection{The differential}

The standard differential acts by edge contraction:
\begin{equation}
d(\Gamma)=\sum_{a\in E(\Gamma)}\Gamma/a
\end{equation}
where $\Gamma$ is a graph from $\bar\mV_v\bar\mE_e\grac^{2}$ or $\bar\mV_v\bar\mE_e\bar\mH_s\grac^{3}$, $E(\Gamma)$ is its set of edges and $\Gamma/a$ is the graph from $\bar\mV_{v-1}\bar\mE_{e-1}\grac^{2}$ or $\bar\mV_{v-1}\bar\mE_{e-1}\bar\mH_s\grac^{3}$ respectively, produced from $\Gamma$ by contracting edge $a$ and merging its end vertices. If a tadpole or a passing vertex is produced, we consider the result to be zero. Also, for oriented and sourced versions, if a closed path is formed or the last source is removed respectively, we consider the result to be zero. The precise signs and verification that the map can be extended to space of invariants and graph complexes $\mO\G_{n}$ and $\mS\G_{n}$ can be found in \cite[Subsection 2.9]{MultiSourced}. In \cite[Subsection 2.10]{MultiSourced} it is shown that $\delta$ is a differential. Exactly the same arguments hold for hairy graph complex $\mH\G_{n}$.

The differential can not produce closed path, so the projection
\begin{equation}
p:\left(\mS\G_n,d\right)\twoheadrightarrow\left(\mO\G_n,d\right)
\end{equation}
is well defined.

On $\mH\G_{n}$ we define another differential $h$, that deletes a hair, summed over all hairs:
\begin{equation}
h(\Gamma)=\sum_{i\in H(\Gamma)}\Gamma/i
\end{equation}
where $\Gamma$ is a graph from $\bar\mV_v\bar\mE_e\bar\mH_s\grac^{3}$, $H(\Gamma)$ is its set of hairs and $\Gamma/i$ is the graph from $\bar\mV_{v}\bar\mE_{e}\bar\mH_{s-1}\grac^{3}$, produced from $\Gamma$ by deleting hair $i$. If it was the last hair or a 2-valent vertex is formed, we consider the result to be zero. We can easily extend it to $\mH\G_{n}$ and see that $h^2=0$, so it is a differential. Also, it easily follows that $hd+hd=0$, so $d+h$ is also a differential.

\subsection{Loop order and the dual differentials}
\label{ss:dual}
So far, we have defined complexes $\left(\mO\G_{n},d\right)$, $\left(\mS\G_{n},d\right)$, $\left(\mH\G_{n},d\right)$, $\left(\mH\G_{n},h\right)$ and $\left(\mH\G_{n},d+h\right)$. 
For a (hairy) graph $\Gamma$, with $e$ edges and $v$ vertices, let the \emph{loop order} of $\Gamma$ be given by $e-v$. The differentials  $d$ and $h$ preserve the loop order, hence all complexes above admit sub-complexes
$$\mB_b\mO \G_n:= \bigoplus_{e-v= b} \mV_v\mE_e \mO \G_n,$$
$$\mB_b\mS \G_n:= \bigoplus_{e-v= b} \mV_v\mE_e \mS \G_n,$$
$$\mB_b\HG_n:= \bigoplus_{e-v= b} \bigoplus_{s\ge 1} \mV_v\mE_e \mH_s G_n,$$
for each $b\in \mathbb{Z}.$ It is clear that
$$
\mO \G_{n} = \bigoplus_{b\ge 0} \mB_b\mO \G_{n},
\quad \mS \G_{n} = \bigoplus_{b\ge 0} \mB_b\mS \G_{n},
\quad \mH \G_n = \bigoplus_{b\ge 0} \mB_b\mH \G_n.
$$
Furthermore, complexes $\mB_b\mO \G_n, \mB_b\mS \G_n$, $ \mB_b\HG_n$ are finite dimensional in each homological degree $k$. Hence there are canonical isomorphisms of vector spaces to their duals
$$
\mB_b\mO \G_n \to \hom(\mB_b\mO \G_n,\Bbbk)=:\mB_b\mO \GC_n,
$$
$$
\mB_b \mS \G_n \to \hom(\mB_b\mS \G_n,\Bbbk)=:\mB_b \mS \GC_n,
$$
$$
\mB_b\HG_n \to \hom(\mB_b\HG_n,\Bbbk)=:\mB_b\HGC_n,
$$
identifying a single term graph $\Gamma$ to the linear map that maps $\Gamma$ to $1$ and all other graphs to $0$.

Any differential $\Delta$ on $\G= \mO \G_n, \mS\G_n$, or $\HG$, that preserves the loop order, is paired with a dual differential $\nabla\leftrightarrow \Delta$ on the dual space $\GC= \mO \GC_n, \mS\GC_n$, or $\HGC$ such that
$$(\GC,\nabla) \cong \prod_{b} \left( \hom(\mB_b G,\Bbbk),\Delta^* \right).$$
We denote the dual differentials of $d$ and $h$ by
\begin{equation}
\delta \leftrightarrow d, \quad
\chi \leftrightarrow h.
\end{equation}

The differential $\delta$ splits a vertex in all possible ways, while $\chi$ adds a hair in all possible ways. The vertex splitting differential (and adding a hair differential) are often defined as the standard differential(s). For hairy graph complexes both $\delta$ and $\chi$ are defined in \cite{DGC2}.

The dual of the projection $p:\left(\mS\G_n,d\right)\twoheadrightarrow\left(\mO\G_n,d\right)$ is the inclusion
\begin{equation}
\iota:\left(\mO\GC_n,\delta\right)\hookrightarrow\left(\mS\GC_n,\delta\right).
\end{equation}

It is a well known result from homological algebra that the cohomology of a dual complex is dual to the homology of the complex. For this reason, we are free to study any of them, in order to obtain the results regarding (co)homology.

\subsection{Skeleton version of directed graph complex}
Instead of directed, oriented and sourced graph complexes $\left(\mD\G_{n},d\right)$, $\left(\mO\G_{n},d\right)$ and $\left(\mS\G_{n},d\right)$, it may sometimes be useful to consider their isomorphic skeleton versions: $\left(\mD^{sk}\G_{n},d\right)$, $\left(\mO^{sk}\G_{n},d\right)$ and $\left(\mS^{sk}\G_{n},d\right)$.

They are defined using theory from \cite[Section 2]{MultiSourced} as follows:
Consider the set of directed connected graphs $\bar\mV_v\bar\mE_e\grac^{3}$ with $v>0$ distinguishable vertices and $e\geq 0$ distinguishable directed edges, all vertices being at least $3$-valent, without tadpoles. In this context,
those graphs are called \emph{core graphs} because we are going to attach edge types to them.
Let $\bar\mV_v\bar\mE_e\Grac^{3}:=\langle\bar\mV_v\bar\mE_e\grac^{3}\rangle$ be the generated vector space.
To each edge of a core graph $\Gamma\in\bar\mV_v\bar\mE_e\grac^{3}$, we attach an element of graded $\langle \sym_2\rangle$ module $\Sigma$ spanned by $\{\Ed[n-1],\dE[n-1],\Ess[n-2]\}$ with $\sym_2$ action
\begin{equation}
\Ed \leftrightarrow\dE,\quad\Ess\mapsto -(-1)^{n}\Ess
\end{equation}
to get $\bar\mV_v\bar\mE_e\Grac^{3}\otimes\Sigma^{\otimes e}$. We say that the type of an edge is its attached element of $\Sigma$.
Let
\begin{equation}
\bar\mV_v\bar\mE_e^{\Sigma}\Grac^{3}:=
\bar\mV_v\bar\mE_e\Grac^{3}\otimes_{S_2^{\times e}}\Sigma^{\otimes e}=
\left(\bar\mV_v\bar\mE_e\Grac^{3}\otimes\Sigma^{\otimes e}\right)_{S_2^{\times e}}
\end{equation}
where $S_2^{\otimes e}$ acts by reversing edges in $\bar\mV_v\bar\mE_e\grac^{3}$, as well as on the element from $\Sigma$ attached to the corresponding edge.

Similarly as before, there are natural right actions of the groups $S_v$ and $S_e$ on $\bar\mV_v\bar\mE_e^{\sym}\Grac^{3}$, where $S_v$ permutes vertices and $S_e$ permutes edges. Let $\sgn_v$ be one-dimensional representations of $S_v$, where the odd permutation reverses the sign. The sign of the action of $S_e$ depends on the types of the permuted edges, such that switching edges with odd degree types change the sign, c.f.\ \cite[Subsection 2.7]{MultiSourced}. Then let
\begin{equation}
\mV_v\mE_e\mD^{sk}\G_{n}:=\left\{
\begin{array}{ll}
\left(\bar\mV_v\bar\mE_e^{\sym}\Grac^{3}\right)^{\sym_v\times\sym_e}[n-vn]
\qquad&\text{for $n$ even,}\\
\left(\bar\mV_v\bar\mE_e^{\sym}\Grac^{3}\otimes\sgn_v\right)^{\sym_v\times\sym_e}[n-vn]
\qquad&\text{for $n$ odd.}
\end{array}
\right.
\end{equation}
This means that for even $n$ switching edges of type $\Ed$ or $\dE$ change the sign, while for odd $n$ switching vertices and edges of type $\Ess$ change the sign. Note the degree shift $[n-vn]$ that comes from the degrees of vertices, while the degree of edges of particular type is already included in $\Sigma$.

The \emph{skeleton version of directed graph complex} is
\begin{equation}
\mD^{sk}\G_{n}:=\bigoplus_{v\geq 1,e\geq 0}\mV_v\mE_e\mD^{sk}\G_{n}.
\end{equation}

On graded $\langle\sym_2\rangle$ module $\Sigma$ there is also a differential defined with
\begin{equation}
\Ess\mapsto\Ed-(-1)^n\dE
\end{equation}
that induces the \emph{edge differential} $d_E$ on $\mD^{sk}\G_{n}$, c.f.\ \cite[Subsection 2.6]{MultiSourced}. The core differential $d_C$ comes from contracting edges of type $\Ed$ and $\dE$, c.f.\ \cite[Subsection 2.10]{MultiSourced}. This enables us to define the \emph{combined differential}
\begin{equation}
d:=d_C+(-1)^{n\deg}d_E
\end{equation}
on $\mD^{sk}\G_{n}$ as in \cite[Subsection 2.11]{MultiSourced}.

This skeleton version $\mD^{sk}\G_{n}$ is defined to be isomorphic to the original version $\mD\G_{n}$. In short, 3-valent vertices and 2-valent sources in a single term graph in $\mD\G_{n}$ are called \emph{skeleton vertices}. Strings of edges and vertices between two skeleton vertices that have to be in the set $\{\Ed,\dE,\EdE\}$, are called \emph{skeleton edges}. A corresponding graph in $\mD^{sk}\G_{n}$ is the one with skeleton vertices as vertices and skeleton edges as edges, where $\EdE$ is mapped to $\Ess$. One can check that the degrees and parities are correctly defined, and obtain the following result. C.f.\ \cite[Subsection 3.4]{MultiSourced}. But note that we have more skeleton vertices here because 2-valent sources are also considered skeleton vertices, and therefore there are less skeleton edges.

\begin{prop}
There is an isomorphism of complexes
\begin{equation}
\kappa:\left(\mD\G_n,d\right)\rightarrow\left(\mD^{sk}\G_n,d\right).
\end{equation}
\end{prop}

Since oriented and sourced graph complexes $\left(\mO\G_n,d\right)$ and $\left(\mS\G_n,d\right)$ are both quotient of directed graph complex $\left(\mD\G_n,d\right)$, we can induce skeleton versions of oriented and sourced graph complexes as follows.

\begin{defprop}
\label{defprop:sk}
Skeleton versions of oriented and sourced graph complexes are
\begin{equation}
\mO^{sk}\G_n:=\kappa\left(\mO\G_n\right),
\end{equation}
\begin{equation}
\mS^{sk}\G_n:=\kappa\left(\mS\G_n\right),
\end{equation}
with the differential $d$ as on $\mD^{sk}\G_n$, where forbidden graphs are considered zero.

The quotient maps, by abuse of notation again denoted by $\kappa$, are isomorphisms of complexes.
\begin{equation}
\kappa:\left(\mO\G_n,d\right)\rightarrow\left(\mO^{sk}\G_n,d\right),
\end{equation}
\begin{equation}
\kappa:\left(\mS\G_n,d\right)\rightarrow\left(\mS^{sk}\G_n,d\right).
\end{equation}
\end{defprop}

\section{The map $\Phi$}
\label{s:map}

In this section we are going to define the main map $\Phi:\left(\HG_{n},d+h\right)\rightarrow\left(\mO\GC_{n+1},\delta\right)$. Thanks to Definition/Proposition \ref{defprop:sk}, it suffices to give
\begin{equation}
\Phi:\left(\HG_{n},d+h\right)\rightarrow\left(\mO^{sk}\G_{n+1},d\right).
\end{equation}
The map 
\begin{equation*}
G: \mO \GC_{n+1}  \to \HGC_{n},
\end{equation*}
from Theorem \ref{thm:main} is the map dual of $\Phi$.


\subsection{Forests}
Let us pick the number of vertices $v$, the number of edges $e$ and the number of hairs $s$.
Let $\Gamma\in\bar\mV_v\bar\mE_e\bar\mH_s\G_{n}$ be a single term graph.

A \emph{forest} is any sub-graph of $\Gamma$ that contains all its hairs (and thus all vertices with hairs), that does not contain cycles (of any orientation), and whose every connected component has exactly one hair. Let a \emph{spanning forest} be a forest that contains all vertices.
Let $F(\Gamma)$ be the set of all spanning forests of $\Gamma$.
An example of a spanning forest is given in Figure \ref{fig:span}.

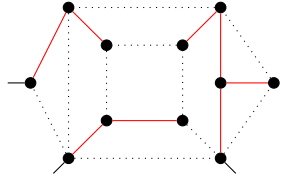
\begin{figure}[H]
$$
\begin{tikzpicture}
 \node[int] (a) at (0,0) {};
 \node[int] (b) at (1,0) {};
 \node[int] (c) at (1,1) {};
 \node[int] (d) at (0,1) {};
 \node[int] (a1) at (-.5,-.5) {};
 \node[int] (b1) at (1.5,-.5) {};
 \node[int] (c1) at (1.5,1.5) {};
 \node[int] (d1) at (-.5,1.5) {};
 \node[int] (x) at (-1,.5) {};
 \node[int] (y) at (1.5,.5) {};
 \node[int] (z) at (2.2,.5) {};
 \draw (a) edge[red] (b);
 \draw (b) edge[dotted] (c);
 \draw (c) edge[dotted] (d);
 \draw (d) edge[dotted] (a);
 \draw (a1) edge[dotted] (b1);
 \draw (b1) edge[red] (y);
 \draw (y) edge[red] (c1);
 \draw (b1) edge[dotted] (z);
 \draw (z) edge[dotted] (c1);
 \draw (y) edge[red] (z);
 \draw (c1) edge[dotted] (d1);
 \draw (d1) edge[dotted] (a1);
 \draw (a) edge[red] (a1);
 \draw (b) edge[dotted] (b1);
 \draw (c) edge[red] (c1);
 \draw (d) edge[red] (d1);
 \draw (a1) edge[dotted] (x);
 \draw (x) edge[red] (d1);
 \draw (x) edge (-1.3,.5);
 \draw (a1) edge (-.7,-.7);
 \draw (b1) edge (1.7,-.7);
\end{tikzpicture}
$$
\caption{\label{fig:span}
An example of a hairy graph and a spanning forest. Edges of the forest are red, while other edges are dotted.}
\end{figure}

\subsection{Model pairs}

For a chosen spanning forest $\tau\in F(\Gamma)$, our first goal is to define $\Phi_{\tau}(\Gamma)\in\bar\mV_v\bar\mE_e\mO^{sk}\G_{n+1}$. As our final goal is to define a map
$$
\Phi: \mV_v\mE_e\mH_s\G_{n} \to \mV_v \mE_e\mO^{sk}\G_{n+1},
$$
the definition of $\Phi_{\tau}(\Gamma)$ must be invariant of the action of $\sym_v\times \sym_s\times \left(\sym_e\ltimes \sym_2^{\times e}\right)$. We will only define $\Phi_{\tau}(\Gamma)$ on some pairs $(\Gamma,\tau)$ that are called \emph{models}, and we extend the definition to all pairs $(\Gamma,\tau)$ by the requirement that the map is invariant under the symmetry action. In order to do that correctly, we need to check two conditions:
\begin{enumerate}
\item[(i)] if there is an element of $\sym_v\times \sym_s\times \left(\sym_e\ltimes \sym_2^{\times e}\right)$ that sends one model to another model, the definition is invariant under its action, and
\item[(ii)] for every pair $(\Gamma,\tau)$, there is an element of $\sym_v\times \sym_s\times \left(\sym_e\ltimes \sym_2^{\times e}\right)$ that sends it to a model.
\end{enumerate}

Note that the spanning forest $\tau$ has $v-s$ edges.
For a pair $(\Gamma,\tau)$ to be a model we require that:
\begin{itemize}
\item edges of a connected component in $\tau$ are directed away from the hair of that connected component;
\item an edge in $\tau$ has the same label as the vertex it is heading to, labels being in the set $\{1,\dots,v-s\}$;
\item a hair labelled by $x$ stands on the hairy vertex labelled by $x+v-s$.
\end{itemize}
An example of a model is given in Figure \ref{fig:model}.

\begin{figure}[H]
$$
\begin{tikzpicture}[scale=1.6]
 \node[int] (a) at (0,0) {};
 \node[below right] at (a) {\bf 1};
 \node[int] (b) at (1,0) {};
 \node[below left] at (b) {\bf 2};
 \node[int] (c) at (1,1) {};
 \node[above left] at (c) {\bf 3};
 \node[int] (d) at (0,1) {};
 \node[above right] at (d) {\bf 4};
 \node[int] (a1) at (-.5,-.5) {};
 \node[below right] at (a1) {\bf 10};
 \node[int] (b1) at (1.5,-.5) {};
 \node[below left] at (b1) {\bf 11};
 \node[int] (c1) at (1.5,1.5) {};
 \node[above left] at (c1) {\bf 5};
 \node[int] (d1) at (-.5,1.5) {};
 \node[above right] at (d1) {\bf 6};
 \node[int] (x) at (-1,.5) {};
 \node[above left] at (x) {\bf 9};
 \node[int] (y) at (1.5,.5) {};
 \node[left] at (y) {\bf 7};
 \node[int] (z) at (2.2,.5) {};
 \node[right] at (z) {\bf 8};
 \draw (a) edge[red,->] node[above] {\color{black} 2} (b);
 \draw (b) edge[dotted,->] node[left] {9} (c);
 \draw (c) edge[dotted,->] node[below] {10} (d);
 \draw (d) edge[dotted,->] node[right] {11} (a);
 \draw (a1) edge[dotted,->] node[below] {14} (b1);
 \draw (b1) edge[red,->] node[right] {\color{black} 7} (y);
 \draw (y) edge[red,->] node[right] {\color{black} 5} (c1);
 \draw (b1) edge[dotted,->] node[right] {16} (z);
 \draw (z) edge[dotted,->] node[right] {17} (c1);
 \draw (y) edge[red,->] node[above] {\color{black} 8} (z);
 \draw (c1) edge[dotted,->] node[above] {12} (d1);
 \draw (d1) edge[dotted,->] node[right] {13} (a1);
 \draw (a) edge[red,<-] node[above] {\color{black} 1} (a1);
 \draw (b) edge[dotted,->] node[above] {15} (b1);
 \draw (c) edge[red,<-] node[below] {\color{black} 3} (c1);
 \draw (d) edge[red,<-] node[below] {\color{black} 4} (d1);
 \draw (a1) edge (-.7,-.7);
 \node[left] at (-.7,-.7) {\bf 2};
 \draw (a1) edge[dotted,->] node[left] {18} (x);
 \draw (x) edge[red,->] node[left] {\color{black} 6} (d1);
 \draw (x) edge (-1.3,.5);
 \node[left] at (-1.3,.5) {\bf 1};
 \draw (b1) edge (1.7,-.7);
 \node[right] at (1.7,-.7) {\bf 3};
\end{tikzpicture}
$$
\caption{\label{fig:model}
An example of model with the spanning forest from Figure \ref{fig:span}. Edges of the forest are red, while other edges are dotted. Labels of vertices and hairs are thick.}
\end{figure}

It is clear that every pair $(\Gamma,\tau)$ can be using an element of $\sym_v\times \sym_s\times \left(\sym_e\ltimes \sym_2^{\times e}\right)$ mapped to a model, so the condition (ii) is fulfilled.

\subsection{Defining the map for a model}

Let us now pick up a model $(\Gamma,\tau)$, consisting of a single term graph $\Gamma\in\bar\mV_v\bar\mE_e\bar\mH_s\G_{n}$ and $\tau\in F(\Gamma)$. A graph $\phi(\Gamma)$ obtained from $\Gamma$ by deleting all hairs and ignoring the degree is a core graph in $\bar\mV_v\bar\mE_e\Grac^3$. To its edges that belong to $E(\tau)$ we attach an edge type $\Ed$, and to those that do not belong to $E(\tau)$ we attach edge type $\Ess$ to get an element of $\bar\mV_v\bar\mE_e^{\sym}\Grac^{3}$, and then after taking coinvariants and adding the degrees an element
\begin{equation}
\Phi_{\tau}(\Gamma)\in\mV_v\mE_e\mD^{sk}\G_{n+1}.
\end{equation}

It is straightforward to check the following:
\begin{itemize}
\item the map is well defined in a sense that if there is an element of $\sym_v\times \sym_s\times \left(\sym_e\ltimes \sym_2^{\times e}\right)$ that sends one model to another model, the same result in $\mV_v\mE_e\mD^{sk}\GC_{n+1}$ is obtained, i.e.\ condition (i) from above is fulfilled;
\item The degree of $\Phi_{\tau}(\Gamma)$ is by 1 greater that the degree of $\Gamma$;
\item The result $\Phi_{\tau}(\Gamma)$ is oriented, i.e.\ it is an element of $\mV_v\mE_e\mO^{sk}\GC_{n+1}$.
\end{itemize}

An example of $\Phi_{\tau}(\Gamma)$ is given in Figure \ref{fig:Phi}.

\begin{figure}[H]
$$
\begin{tikzpicture}
 \node[int] (a) at (0,0) {};
 \node[int] (b) at (1,0) {};
 \node[int] (c) at (1,1) {};
 \node[int] (d) at (0,1) {};
 \node[int] (a1) at (-.5,-.5) {};
 \node[int] (b1) at (1.5,-.5) {};
 \node[int] (c1) at (1.5,1.5) {};
 \node[int] (d1) at (-.5,1.5) {};
 \node[int] (x) at (-1,.5) {};
 \node[int] (y) at (1.5,.5) {};
 \node[int] (z) at (2.2,.5) {};
 \draw (a) edge[-latex] (b);
 \draw (b) edge[crossed,->] (c);
 \draw (c) edge[crossed,->] (d);
 \draw (d) edge[crossed,->] (a);
 \draw (a1) edge[crossed,->] (b1);
 \draw (b1) edge[-latex] (y);
 \draw (y) edge[-latex] (c1);
 \draw (b1) edge[crossed,->] (z);
 \draw (z) edge[crossed,->] (c1);
 \draw (y) edge[-latex] (z);
 \draw (c1) edge[crossed,->] (d1);
 \draw (d1) edge[crossed,->] (a1);
 \draw (a) edge[latex-] (a1);
 \draw (b) edge[crossed,->] (b1);
 \draw (c) edge[latex-] (c1);
 \draw (d) edge[latex-] (d1);
 \draw (a1) edge[crossed,->] (x);
 \draw (x) edge[-latex] (d1);
\end{tikzpicture}
$$
\caption{\label{fig:Phi}
Oriented graph $\Phi_\tau(\Gamma)$ for the graph $\Gamma$ and spanning forest $\tau$ from  Figure \ref{fig:span}.}
\end{figure}

\subsection{The final map}

The map is now extended to all pairs $(\Gamma,\tau)$ by invariance under the action of $\sym_v\times \sym_s\times \left(\sym_e\ltimes \sym_2^{\times e}\right)$. Then let us define
\begin{equation}
\Phi:\bar\mV_v\bar\mE_e\bar\mH_s\G_{n}\rightarrow\mV_v\mE_e\mO^{sk}\G_{n+1}, \quad
\Gamma\mapsto\sum_{\tau\in F(\Gamma)}\Phi_{\tau}(\Gamma).
\end{equation}
The invariance under all actions implies that the induced map $\Phi:\mV_v\mE_e\mH_s\G_{n}\rightarrow\mV_v\mE_e\mO^{sk}\G_{n+1}$ is well defined. It is then extended to the whole
\begin{equation}
\Phi:\mH\G_{n}\rightarrow\mO^{sk}\G_{n+1}.
\end{equation}

\begin{prop}
\label{prop:map}
The map $\Phi:\left(\HG_{n},d+h\right)\rightarrow\left(\mO^{sk}\GC_{n+1},d\right)$ is a map of complexes of degree 1, i.e.\
$$
\Phi((d+h)\Gamma)=d \Phi(\Gamma)
$$
for every $\Gamma\in\HG_{n}$.
\end{prop}
\begin{proof}

We have already checked that the degree of $\Phi$ is 1.

For the other claim let us pick a single term graph $\Gamma\in\bar\mV_v\bar\mE_e\bar\mH_s\G_{n}$.
It holds that
$$
\Phi(d\Gamma)=
\Phi\left(\sum_{a\in E(\Gamma)}\Gamma/a\right)=
\sum_{a\in E(\Gamma)}\Phi\left(\Gamma/a\right)=
\sum_{a\in E(\Gamma)}\sum_{\tau\in F(\Gamma/a)}\Phi_{\tau}(\Gamma/a)
$$
where $\Gamma/a$ is contracting an edge $a$ in $\Gamma$.
Spanning forests of $\Gamma/a$ are in natural bijection with spanning forests of $\Gamma$ that contain $a$, $\tau/a\leftrightarrow\tau$, so we can write
$$
\Phi(d\Gamma)=
\sum_{a\in E(\Gamma)}
\sum_{\substack{\tau\in F(\Gamma)\\ a\in E(\tau)}}\Phi_{\tau/a}(\Gamma/a)=
\sum_{\tau\in F(\Gamma)}\sum_{a\in E(\tau)}\Phi_{\tau/a}(\Gamma/a).
$$

\begin{lemma}
Let $\Gamma\in\bar\mV_v\bar\mE_e\bar\mH_s\G_{n}$, $\tau\in F(\Gamma)$ and $a\in E(\tau)$. Then
\begin{equation}
\Phi_{\tau/a}(\Gamma/a)\sim \Phi_{\tau}(\Gamma)/a,
\end{equation}
where $\sim$ means that they are in the same class of coinvariants under the action of $\sym_v\times\sym_e$.
\end{lemma}
\begin{proof}
It is clear that one side is $\pm$ the other side. Careful calculation of the sign is left to the reader.
\end{proof}



The lemma implies that
\begin{equation}
\Phi(d\Gamma)\sim\sum_{\tau\in F(\Gamma)}\sum_{a\in E(\tau)}\Phi_{\tau}(\Gamma)/a.
\end{equation}

Still on the left-hand-side of the claimed relation we have
$$
\Phi(h(\Gamma))=\Phi\left(\sum_{i\in H(\Gamma)}\Gamma/i\right)
=\sum_{i\in H(\Gamma)}\Phi(\Gamma/i)
=\sum_{i\in H(\Gamma)}\sum_{\tau\in F(\Gamma/i)}\Phi_{\tau}(\Gamma/i),
$$
where $\Gamma/i$ deletes hair $i$.

Spanning forests in $F(\Gamma/i)$ correspond to a kind of forests of $\Gamma$ where one connected component has two hairs. Therefore, let $FD(\Gamma)$ (\emph{double-hair forests}) be the set of all sub-graphs $\lambda$ of $\Gamma$ that contain all vertices and hairs, have no cycles, whose one connected component has exactly two hairs and whose other connected components have exactly one hair. Let those two hairs be $j(\lambda)$ and $k(\lambda)$.
Sets $\{(i,\tau)|i\in H(\Gamma),\tau\in F(\Gamma/i)\}$ and $\{(\lambda,i)|\lambda\in FD(\Gamma),i\in\{j(\lambda),k(\lambda)\}\}$ are clearly bijective.
So we have
\begin{equation}
\label{eq:map2}
\Phi(\Gamma/i)=
\sum_{\lambda\in FD(\Gamma)}\left(\Phi_{\lambda}(\Gamma/j(\lambda))+\Phi_{\lambda}(\Gamma/k(\lambda))\right).
\end{equation}

On the other side the differential is $d=d_C\pm d_E$. It acts on edges of $\Phi(\Gamma)$. They come from edges in $\Gamma$ and can be split into three sets:
\begin{itemize}
\item edges that are in $\tau$ form the set $E(\tau)$;
\item edges that connect two connected components of $\tau$ form the set $ED(\tau)$;
\item edges that make a cycle in a connected component of $\tau$ form the set $EC(\tau)$.
\end{itemize}

Only edges of type $\Ed$ and $\dE$ can be contracted by $d_C$. They come from the edges in $E(\tau)$ so
\begin{equation}
\label{eq:map1}
d_C \Phi(\Gamma)=
d_C\left(\sum_{\tau\in F(\Gamma)}\Phi_{\tau}(\Gamma)\right)=
\sum_{\tau\in F(\Gamma)}d_C\left(\Phi_{\tau}(\Gamma)\right)=
\sum_{\tau\in F(\Gamma)}\sum_{a\in E(\tau)}\Phi_{\tau}(\Gamma)/a\sim
\Phi(d\Gamma).
\end{equation}

The edge differential $d_E$ acts on edges of type $\Ess$, which are those in the sets $ED(\tau)$ and $EC(\tau)$. We then split
\begin{equation}
d_E(\Phi_\tau(\Gamma))=d_{ED}(\Phi_\tau(\Gamma))+d_{EC}(\Phi_\tau(\Gamma)),
\end{equation}
where
\begin{equation}
d_{ED}(\Phi_\tau(\Gamma))=\sum_{a\in ED(\tau)}d_E^{(a)}(\Phi_\tau(\Gamma)),\quad
d_{EC}(\Phi_\tau(\Gamma))=\sum_{a\in EC(\tau)}d_E^{(a)}(\Phi_\tau(\Gamma)),
\end{equation}

where $d_E^{(a)}$ maps edge $a$ as $\Ess\mapsto=\Ed-(-1)^{n+1}\dE$.


\begin{lemma}
\label{lem:map3}
Let $\Gamma\in\bar\mV_v\bar\mE_e\bar\mH_s\G_{n}$ be a single term graph. Then
$$
\sum_{\tau\in F(\Gamma)}d_{EC}\left(\Phi_{\tau}(\Gamma)\right)\sim 0.
$$
\end{lemma}
\begin{proof}
Let
$$
N(\Gamma):=
\sum_{\tau\in F(\Gamma)}d_{EC}\left(\Phi_{\tau}(\Gamma)\right)=
\sum_{\tau\in F(\Gamma)}\sum_{a\in EC(\tau)}d_E^{(a)}\left(\Phi_{\tau}(\Gamma)\right).
$$

Terms in the above relation can be summed in another order. Let $FC(\Gamma)$ (\emph{cycled forests}) be the set of all sub-graphs $\rho$ of $\Gamma$ that contain all vertices and hairs, have $v-s+1$ edges, and whose every connected component has exactly one hair. Those graphs are similar to spanning forests, but have one cycle.
Let $C(\rho)$ be the set of edges in the cycle of $\rho$. Clearly, $\rho\setminus \{a\}$ for $a\in C(\rho)$ is a spanning forest of $\Gamma$ and sets $\{(\tau,a)|\tau\in F(\Gamma),a\in EC(\tau)\}$ and $\{(\rho,a)|\rho\in FC(\Gamma),a\in C(\rho)\}$ are bijective, so
$$
N(\Gamma)=\sum_{\rho\in FC(\Gamma)}\sum_{a\in C(\rho)}d_E^{(a)}\left(\Phi_{\rho\setminus \{a\}}(\Gamma)\right).
$$

It is now enough to show that
$$
\sum_{a\in C(\rho)}d_E^{(a)}\left(\Phi_{\rho\setminus \{a\}}(\Gamma)\right)\sim 0
$$
for every $\rho\in FC(\Gamma)$.
Let $y\in V(\Gamma)$ be the vertex in the cycle of $\rho$ closest to the hair of its connected component (along $\rho$). After choosing $a\in C(\rho)$, the cycle in $\Phi_{\rho\setminus \{a\}}(\Gamma)$ has the edge $a$ of type $\Ess$, and other edges of type $\Ed$ or $\dE$ with direction from $y$ to the edge $a$, such as in the following diagram.
$$
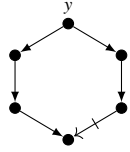
\begin{tikzpicture}[baseline=0ex,scale=.7]
 \node[int] (a) at (0,1.1) {};
 \node[int] (b) at (1,.5) {};
 \node[int] (c) at (1,-.5) {};
 \node[int] (d) at (0,-1.1) {};
 \node[int] (e) at (-1,-.5) {};
 \node[int] (f) at (-1,.5) {};
 \node[above] at (a) {$\scriptstyle y$};
 \draw (a) edge[-latex] (b);
 \draw (b) edge[-latex] (c);
 \draw (a) edge[-latex] (f);
 \draw (f) edge[-latex] (e);
 \draw (e) edge[-latex] (d);
 \draw (c) edge[->,crossed] (d);
\end{tikzpicture}
$$
After acting by $d_E^{(a)}$ this $\Ess$ is replaced by $\Ed+(-1)^n\dE$, such as in the following diagram.
$$
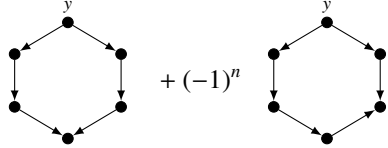
\begin{tikzpicture}[baseline=-.6ex,scale=.7]
 \node[int] (a) at (0,1.1) {};
 \node[int] (b) at (1,.5) {};
 \node[int] (c) at (1,-.5) {};
 \node[int] (d) at (0,-1.1) {};
 \node[int] (e) at (-1,-.5) {};
 \node[int] (f) at (-1,.5) {};
 \node[above] at (a) {$\scriptstyle y$};
 \draw (a) edge[-latex] (b);
 \draw (b) edge[-latex] (c);
 \draw (a) edge[-latex] (f);
 \draw (f) edge[-latex] (e);
 \draw (e) edge[-latex] (d);
 \draw (c) edge[-latex] (d);
\end{tikzpicture}
\quad+(-1)^n\quad
\begin{tikzpicture}[baseline=-.6ex,scale=.7]
 \node[int] (a) at (0,1.1) {};
 \node[int] (b) at (1,.5) {};
 \node[int] (c) at (1,-.5) {};
 \node[int] (d) at (0,-1.1) {};
 \node[int] (e) at (-1,-.5) {};
 \node[int] (f) at (-1,.5) {};
 \node[above] at (a) {$\scriptstyle y$};
 \draw (a) edge[-latex] (b);
 \draw (b) edge[-latex] (c);
 \draw (a) edge[-latex] (f);
 \draw (f) edge[-latex] (e);
 \draw (e) edge[-latex] (d);
 \draw (c) edge[latex-] (d);
\end{tikzpicture}
$$
Careful calculation of the sign shows that those two terms are cancelled with terms given from choosing neighbouring edges in $C(\rho)$, and two last terms which do not have a corresponding neighbour are indeed $0$ as they have a cycle along arrows.
This concludes the proof that $N(\Gamma)=0$.
\end{proof}

The similar study of the action on edges from $ED(\tau)$ leads to the following lemma.

\begin{lemma}
\label{lem:map4}
Let $\Gamma\in\bar\mV_v\bar\mE_e\bar\mH_s\G_{n}$ be a single term graph. Then
$$
\sum_{\tau\in F(\Gamma)}d_{ED}\left(\Phi_{\tau}(\Gamma)\right)\sim
\sum_{\lambda\in FD(\Gamma)}\left(\Phi_{\lambda}\left(\Gamma/j(\lambda)\right)+\Phi_{\lambda}\left(\Gamma/k(\lambda)\right)\right).
$$
\end{lemma}
\begin{proof}
It holds that
$$
\sum_{\tau\in F(\Gamma)}d_{ED}\left(\Phi_{\tau}(\Gamma)\right)=
\sum_{\tau\in F(\Gamma)}\sum_{a\in ED(\tau)}d_E^{(a)}(\Phi_\tau(\Gamma)).
$$

For $\lambda\in FD(\Gamma)$ let $P(\lambda)$ be the set of edges in the path from $j(\lambda)$ to $k(\lambda)$. Clearly, $\lambda\setminus \{a\}$ for $a\in P(\lambda)$ is a spanning forest of $\Gamma$ and $a$ is in $ED(\Gamma)$ for that spanning forest. One can easily see that sets $\{(\tau,a)|\tau\in F(\Gamma),a\in ED(\tau)\}$ and $\{(\lambda,a)|\lambda\in FD(\Gamma),a\in P(\lambda)\}$ are bijective, so
$$
\sum_{\tau\in F(\Gamma)}d_{ED}\left(\Phi_{\tau}(\Gamma)\right)=
\sum_{\lambda\in FD(\Gamma)}\sum_{a\in P(\lambda)}d_E^{(a)}\left(\Phi_{\lambda\setminus\{a\}}(\Gamma)\right).
$$

To finish the proof it is enough to show that
$$
\sum_{a\in P(\lambda)}d_E^{(a)}\left(\Phi_{\lambda\setminus\{a\}}(\Gamma)\right)\sim
\Phi_{\lambda}\left(\Gamma/j(\lambda)\right)+\Phi_{\lambda}\left(\Gamma/k(\lambda)\right)
$$
for every $\lambda\in FD(\Gamma)$.
After choosing $a\in P(\lambda)$ the path from $j(\lambda)$ to $k(\lambda)$ along $\lambda$ in $\Phi_{\lambda\setminus\{a\}}(\Gamma)$ has the edge $a$ of type $\Ess$, and the other edges of type $\Ed$ or $\dE$ with direction from $j(\lambda)$ or $k(\lambda)$ to the edge $a$, such as in the following diagram.
$$
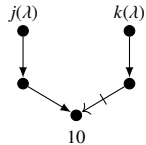
\begin{tikzpicture}[baseline=0ex,scale=.7]
 \node[int] (b) at (1,.5) {};
 \node[int] (c) at (1,-.5) {};
 \node[int] (d) at (0,-1.1) {};
 \node[int] (e) at (-1,-.5) {};
 \node[int] (f) at (-1,.5) {};
 \node[above] at (b) {$\scriptstyle k(\lambda)$};
 \node[above] at (f) {$\scriptstyle j(\lambda)$};
 \draw (b) edge[-latex] (c);
 \draw (f) edge[-latex] (e);
 \draw (e) edge[-latex] (d);
 \draw (c) edge[->,crossed] (d);
\end{tikzpicture}
$$
After acting by $d_E^{(a)}$ this $\Ess$ is replaced by $\Ed+(-1)^n\dE$, such as in the following diagram.
$$
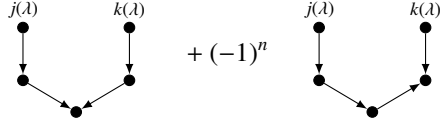
\begin{tikzpicture}[baseline=-.6ex,scale=.7]
 \node[int] (b) at (1,.5) {};
 \node[int] (c) at (1,-.5) {};
 \node[int] (d) at (0,-1.1) {};
 \node[int] (e) at (-1,-.5) {};
 \node[int] (f) at (-1,.5) {};
 \node[above] at (b) {$\scriptstyle k(\lambda)$};
 \node[above] at (f) {$\scriptstyle j(\lambda)$};
 \draw (b) edge[-latex] (c);
 \draw (f) edge[-latex] (e);
 \draw (e) edge[-latex] (d);
 \draw (c) edge[-latex] (d);
\end{tikzpicture}
\quad+(-1)^n\quad
\begin{tikzpicture}[baseline=-.6ex,scale=.7]
 \node[int] (b) at (1,.5) {};
 \node[int] (c) at (1,-.5) {};
 \node[int] (d) at (0,-1.1) {};
 \node[int] (e) at (-1,-.5) {};
 \node[int] (f) at (-1,.5) {};
 \node[above] at (b) {$\scriptstyle k(\lambda)$};
 \node[above] at (f) {$\scriptstyle j(\lambda)$};
 \draw (b) edge[-latex] (c);
 \draw (f) edge[-latex] (e);
 \draw (e) edge[-latex] (d);
 \draw (c) edge[latex-] (d);
\end{tikzpicture}
$$
Careful calculation of the sign shows that those two terms are cancelled with terms given from choosing neighbouring edges in $P(\lambda)$. The
two last terms which does not have corresponding neighbour are exactly $\Phi_{\lambda}\left(\Gamma/j(\lambda)\right)$ and $\Phi_{\lambda}\left(\Gamma/k(\lambda)\right)$, which was to be demonstrated.
\end{proof}

Equations \eqref{eq:map1} and \eqref{eq:map2}, and Lemmas \ref{lem:map3} and \ref{lem:map4} imply that
\begin{multline*}
\Phi(d(\Gamma))+\Phi(h(\Gamma))\sim
d_C(\Phi(\Gamma))+\sum_{\lambda\in FD(\Gamma)}\left(\Phi_{\lambda}(h_{j(\lambda)}(\Gamma))+\Phi_{\lambda}(h_{k(\lambda)}(\Gamma))\right)\sim\\
\sim d_C(\Phi(\Gamma))+
\sum_{\tau\in F(\Gamma)}d_{EC}\left(\Phi_{\tau}(\Gamma)\right)+
\sum_{\tau\in F(\Gamma)}d_{ED}\left(\Phi_{\tau}(\Gamma)\right)=
d_C(\Phi(\Gamma))+d_E(\Phi(\Gamma))=d(\Phi(\Gamma)).
\end{multline*}

After taking coinvariants this implies that
$$
\Phi(d(\Gamma))+\Phi(h(\Gamma))=d(\Phi(\Gamma))
$$
for each $\Gamma\in\HG_{n}$. Hence, $\Phi:(\HG_{n},d+h)\to \mO (\G_{n+1},d)$ is a map of complexes.
\end{proof}

\subsection{Differential grading on the number of hairs/sources}

Note that the part of $\HG_{n}$ with $s$ hairs is mapped to the part of $\mO\G_{n+1}$ with $s$ sources, so the map $\Phi:\mH\G_{n}\rightarrow\mO\G_{n+1}$ can be split as
\begin{equation}
\Phi:\mH_s\G_{n}\rightarrow\mS_s\mO\G_{n+1},
\end{equation}
where $\mS_s\mO\G_{n+1}$ is the part of $\mO\G_{n+1}$ spanned by oriented graphs with exactly $s$ sources.

There are filtrations on $\left(\mH\G_{n},d+h\right)$ and $\left(\mO\G_{n+1},d\right)$ on the number of hairs and sources respectively. Their differential gradings are $\left(\mH\G_{n},d\right)$ and $\left(\mO\G_{n+1},d_0\right)$, where $d_0$ is the part of $d$ that does not destroy a source. We will refer to graph complexes with such a differential $d_0$ as \emph{fixed source graph complexes}.

As the map $\Phi$ maps a graph with $s$ hairs to a sum of graphs all containing $s$ sources, the same map of vector spaces, $\Phi$, is also a map of complexes
\begin{equation}
\Phi:\left(\mH\G_{n},d\right)\rightarrow\left(\mO\G_{n+1},d_0\right).
\end{equation}

\subsection{The dual map}
As the map $\Phi$ preserves the loop order, we obtain a dual map $\Phi\leftrightarrow G$
\begin{equation}
G:\left(\mO\GC_{n+1},\delta\right)\rightarrow\left(\mH\GC_{n},\delta+\chi\right).
\end{equation}
Similarly, the same map is a map of complexes
\begin{equation}
G:\left(\mO\GC_{n+1},\delta_0\right)\rightarrow\left(\mH\GC_{n},\delta\right),
\end{equation}
where $\delta_0\leftrightarrow d_0$ is the part of $\delta$ that does not produce a source by splitting non-source.

Let $\Gamma\in  \mO \GC_{n+1}$ be a single term graph. We define the \emph{hairy skeleton} $hs(\Gamma)\in \mH\GC_{n}$ to be the hairy graph obtained from $\Gamma$ by:
\begin{itemize}
\item putting a hair on each source vertex;
\item contracting each occurrence of bi-valent target vertices $\EdE$ into a single non-oriented edge;
\end{itemize}
It is evident that
 \begin{equation}\label{eq:G}
G(\Gamma) = \begin{cases} hs(\Gamma) & \text{if $\Gamma=\Phi_\tau(hs(\Gamma))$ for some spanning forest $\tau$} \\
0& \text{otherwise.}
\end{cases}
\end{equation}
\begin{prop}
Let $\Gamma\in  \mO \GC_{n+1}$ be a single term graph whose hairy skeleton $hs(\Gamma)$ contains $v$ vertices, $e$ edges, and $s$ sources . We have that $\Gamma=\Phi_\tau(hs(\Gamma))$ for some spanning forest $\tau$ if and only if $\Gamma$ contains $e-v+s$  bi-valent target vertices $\EdE$.
\end{prop}
\begin{proof}
If $\Gamma=\Phi_\tau(hs(\Gamma))$, it is clear that $\Gamma$ must contain $e-v+s$ of  bi-valent target vertices $\EdE$. Now, let $\Gamma$ be an oriented graph that contains $e-v+s$ bi-valent target vertices $\EdE$. Let $\gamma$ be the subgraph of $\Gamma$ obtained by removing all bi-valent targets $\EdE$ from $\Gamma$. Note that $\gamma$ is a graph with $v$ vertices and $v-s$ edges. Such a graph must contain at least $s$ different connected components, with equality if and only if $\gamma$ is a forest. Next, note that $\gamma$ is oriented, hence each component must contain a source vertex, and $\gamma$ contains precisely $s$ source vertices.

We conclude that $\gamma$ must be a forest with $s$ components, where each component contains precisely one source vertex. Hence, we get that
$$\Gamma= \Phi_{hs(\gamma)}(hs(\Gamma)).$$
\end{proof}
This proposition gives us an alternative description of the map $G$.

 \begin{equation}\label{eq:Ge}
G(\Gamma) = \begin{cases} hs(\Gamma) & \text{if $\Gamma$ contains $v-e+s$ bi-valent target vertices $\EdE$} \\
0& \text{otherwise.}
\end{cases}
\end{equation}

\section{The map is a quasi-isomorphism}
\label{s:qi}

\begin{thm}
\label{thm:main1}
The map 
$$
\Phi:  \left(\HG_n,d\right)\to\left( \mO \G_{n+1},d_0\right)
$$
and the projection
$$
p:\left(\mS\G_n,d_0\right)\twoheadrightarrow\left(\mO\G_n,d_0\right)
$$
are quasi-isomorphism.
\end{thm}
\begin{proof}
We prove both claims simultaneously. After splitting by the number of hairs/sources, we need to prove that
$$
\Phi:  \left(\mH_s\G_n,d\right)\to\left( \mS_s\mO \G_{n+1},d_0\right)
\twoheadleftarrow\left( \mS_s\G_{n+1},d_0\right)
$$
are quasi-isomorphisms. Using Definition/Proposition \ref{defprop:sk} it is enough to prove the claim for
$$
\Phi:  \left(\mH_s\G_n,d\right)\to\left( \mS_s\mO^{sk} \G_{n+1},d_0\right)
\twoheadleftarrow\left( \mS_s^{sk}\G_{n+1},d_0\right)
$$
where
\begin{equation}
\mS_s\mO^{sk}\G_n:=\kappa\left(\mS_s\mO\G_n\right),
\end{equation}
\begin{equation}
\mS_s^{sk}\G_n:=\kappa\left(\mS_s\G_n\right).
\end{equation}
On the mapping cone of them we set up the spectral sequence on the number of vertices. Standard splitting of complexes as the product of complexes with fixed loop number implies the correct convergence. It is therefore enough to prove the claim for the first differential of the spectral sequence (differential grading).

The edge differential does not change the number of vertices, while the core differential does. Therefore, on the first page of the spectral sequences there are mapping cones of the maps
$$
\Phi:  \left(\mH_s\G_n,0\right)\to\left( \mS_s\mO^{sk} \G_{n+1},d_{E0}\right)
\twoheadleftarrow\left( \mS_s^{sk}\G_{n+1},d_{E0}\right),
$$
where $d_{E0}$ is the part of the edge differential $d_E$ that does not destroy a source.

Complexes are direct sums of those with fixed number of vertices and edges, so it is enough to show the claim for
$$
\Phi:  \left(\mV_v\mE_e\mH_s\G_n,0\right)\to
\left(\mV_v\mE_e\mS_s\mO^{sk} \G_{n+1},d_{E0}\right)
\twoheadleftarrow\left( \mV_v\mE_e\mS_s^{sk}\G_{n+1},d_{E0}\right),
$$
where $\mV_v\mE_e\mS_s\mO^{sk} \G_{n+1}$ is the part of $\mV_v\mE_e\mO^{sk} \G_{n+1}$ with $s$ sources.

Recall that the skeleton versions of oriented and sourced graph complexes $\mV_v\mE_e\mS_s\mO^{sk} \G_{n+1}$ and $\mV_v\mE_e\mS_s^{sk} \G_{n+1}$ are spaces of invariants of the action of $\sym_v\times\sym_e$, while the hairy graph complex $\mV_v\mE_e\mH_s\G_n$ is the space of invariants of the action of $\sym_v\times \sym_s\times \left(\sym_e\ltimes \sym_2^{\times e}\right)$. In order to have the same group acting in a latter case, let us redefine the hairy graph complex in two steps: Let
\begin{equation}
\bar\mV_v\dot\mE_e\mH_s\G_{n}:=\left\{
\begin{array}{ll}
\left(\bar\mV_v\bar\mE_e\bar\mH_s\G_n\otimes\sgn_s\right)^{\sym_s\times\sym_2^{\times e}}
\qquad&\text{for $n$ even,}\\
\left(\bar\mV_v\bar\mE_e\bar\mH_s\G_n\otimes\sgn_s\otimes\sgn_2^{\otimes e}\right)^{\sym_s\times\sym_2^{\times e}}
\qquad&\text{for $n$ odd,}
\end{array}
\right.
\end{equation}
and then
\begin{equation}
\mV_v\mE_e\mH_s\G_{n}=\left\{
\begin{array}{ll}
\left(\bar\mV_v\dot\mE_e\mH_s\G_{n}\otimes\sgn_e\right)^{\sym_v\times\sym_e}
\qquad&\text{for $n$ even,}\\
\left(\bar\mV_v\dot\mE_e\mH_s\G_{n}\otimes\sgn_v\right)^{\sym_v\times\sym_e}
\qquad&\text{for $n$ odd.}
\end{array}
\right.
\end{equation}

The action of $\sym_v\times\sym_e$ clearly commutes with the map $\Phi$. Since the edge differential does not change the number of vertices and edges, taking homology commutes with taking coinvariants of that action. Therefore, it is now enough to show the claim for
$$
\Phi:  \left(\bar\mV_v\dot\mE_e\mH_s\G_{n},0\right)\to
\left(\bar\mV_v\bar\mE_e\mS_s\mO^{sk} \G_{n+1},d_{E0}\right)
\twoheadleftarrow\left(\bar\mV_v\bar\mE_e\mS_s^{sk}\G_{n+1},d_{E0}\right).
$$

Let us pick up a particular single term graph $\Gamma\in\bar\mV_v\dot\mE_e\mH_s\G_{n}$. Note that in $\Gamma$ edges are (up to the sign) undirected. Let $\langle\mO\Gamma\rangle$ be the subspace of $\bar\mV_v\bar\mE_e\mS_s\mO^{sk} \G_{n+1}$ spanned by oriented graphs of the shape $\Gamma$ without hairs, but with source exactly where the hairs were. Similarly, let $\langle\mS\Gamma\rangle$ be the subspace of $\bar\mV_v\bar\mE_e\mS_s{sk} \G_{n+1}$ spanned by sourced graphs of the shape $\Gamma$ without hairs, but with source exactly where the hairs were.

The map $\Phi$ is defined such that $\Phi(\Gamma)\in\langle\mO\Gamma\rangle$. Also, differential $d_{E0}$ acts within particular subspaces $\langle\mO\Gamma\rangle$ and $\langle\mS\Gamma\rangle$. Therefore, we can split the map as a direct sum and it is enough to prove the clam for
\begin{equation}
\label{eq:hreduced}
\Phi:  \left(\langle\Gamma\rangle,0\right)\to
\left(\langle\mO\Gamma\rangle,d_{E0}\right)
\twoheadleftarrow\left(\langle\mS\Gamma\rangle,d_{E0}\right),
\end{equation}
for every $\Gamma\in\bar\mV_v\dot\mE_e\mH_s\G_{n}$.

%

In order to prove that, let us choose $v-s$ edges in $\Gamma$, say $a_1,\dots,a_{v-s}$. Let $F(a_1,\dots,a_i)$ be the sub-graph of $\Gamma$ that includes those edges, all hairs and all necessary vertices. We require that for every $i=1,\dots,v-s$ the sub-graph $F(a_1,\dots,a_i)$ is a forest. Recall that in a forest, every connected component has exactly one hair. Clearly, $F(a_1,\dots, a_{v-s})$ is a spanning forest.

For every $i=0,\dots,v-s$, we form two graph complexes $\langle\mO\Gamma^i\rangle$ and $\langle\mS\Gamma^i\rangle$ as follows: They are all spanned by graphs with a core graph being un-haired $\Gamma$ with attached edge types from dg $\langle \sym_2\rangle$ module $\bar\Sigma$ spanned by $\{\Ed[n],\dE[n],\Ess[n-1],\ET[n]\}$ with $\sym_2$ action
\begin{equation}
\Ed \leftrightarrow\dE,\quad
\Ess\mapsto (-1)^{n}\Ess,\quad
\ET\mapsto (-1)^{n+1}\ET,
\end{equation}
and the differential
\begin{equation}
\Ess\mapsto\Ed-(-1)^n\dE.
\end{equation}

The complex $\langle\mS\Gamma^i\rangle$ is spanned by graphs whose attached edge types fulfill the following conditions:
\begin{enumerate}
\item Edges $a_1,\dots, a_i$ have type $\ET$, and other edges have other types.
\item No vertex in the forest $F(a_1,\dots,a_i)$ has a neighbouring edge of type $\Ed$ or $\dE$ heading towards it.
\item Every vertex outside the forest $F(a_1,\dots,a_i)$ has a neighbouring edge of type $\Ed$ or $\dE$ heading towards it.
\end{enumerate}
The complex $\langle\mO\Gamma^i\rangle$ is spanned by graphs whose attached edge types fulfill the conditions above together with:
\begin{enumerate}
\item[(4)] There are no cycles along arrows on edges of type $\Ed$ and $\dE$.
\end{enumerate}
Examples of graphs in $\langle\mS\Gamma^i\rangle$ and $\langle\mO\Gamma^i\rangle$ are shown in Figure \ref{fig:SGamma}.

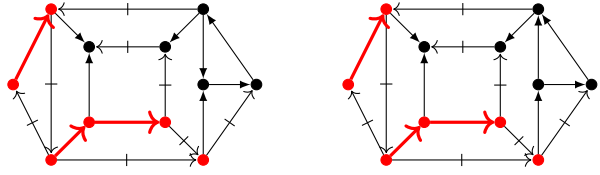
\begin{figure}[H]
$$
\begin{tikzpicture}
 \node[red,int] (a) at (0,0) {};
 \node[red,int] (b) at (1,0) {};
 \node[int] (c) at (1,1) {};
 \node[int] (d) at (0,1) {};
 \node[red,int] (a1) at (-.5,-.5) {};
 \node[red,int] (b1) at (1.5,-.5) {};
 \node[int] (c1) at (1.5,1.5) {};
 \node[red,int] (d1) at (-.5,1.5) {};
 \node[red,int] (x) at (-1,.5) {};
 \node[int] (y) at (1.5,.5) {};
 \node[int] (z) at (2.2,.5) {};
 \draw (a) edge[red,very thick,->] (b);
 \draw (b) edge[crossed,->] (c);
 \draw (c) edge[crossed,->] (d);
 \draw (d) edge[latex-] (a);
 \draw (a1) edge[crossed,->] (b1);
 \draw (b1) edge[-latex] (y);
 \draw (y) edge[latex-] (c1);
 \draw (b1) edge[crossed,->] (z);
 \draw (z) edge[-latex] (c1);
 \draw (y) edge[-latex] (z);
 \draw (c1) edge[crossed,->] (d1);
 \draw (d1) edge[crossed,->] (a1);
 \draw (a) edge[red,very thick,<-] (a1);
 \draw (b) edge[crossed,->] (b1);
 \draw (c) edge[latex-] (c1);
 \draw (d) edge[latex-] (d1);
 \draw (a1) edge[crossed,->] (x);
 \draw (x) edge[red,very thick,->] (d1);
\end{tikzpicture}
\quad\quad\quad
\begin{tikzpicture}
 \node[red,int] (a) at (0,0) {};
 \node[red,int] (b) at (1,0) {};
 \node[int] (c) at (1,1) {};
 \node[int] (d) at (0,1) {};
 \node[red,int] (a1) at (-.5,-.5) {};
 \node[red,int] (b1) at (1.5,-.5) {};
 \node[int] (c1) at (1.5,1.5) {};
 \node[red,int] (d1) at (-.5,1.5) {};
 \node[red,int] (x) at (-1,.5) {};
 \node[int] (y) at (1.5,.5) {};
 \node[int] (z) at (2.2,.5) {};
 \draw (a) edge[red,very thick,->] (b);
 \draw (b) edge[crossed,->] (c);
 \draw (c) edge[crossed,->] (d);
 \draw (d) edge[latex-] (a);
 \draw (a1) edge[crossed,->] (b1);
 \draw (b1) edge[-latex] (y);
 \draw (y) edge[-latex] (c1);
 \draw (b1) edge[crossed,->] (z);
 \draw (z) edge[-latex] (c1);
 \draw (y) edge[-latex] (z);
 \draw (c1) edge[crossed,->] (d1);
 \draw (d1) edge[crossed,->] (a1);
 \draw (a) edge[red,very thick,<-] (a1);
 \draw (b) edge[crossed,->] (b1);
 \draw (c) edge[latex-] (c1);
 \draw (d) edge[latex-] (d1);
 \draw (a1) edge[crossed,->] (x);
 \draw (x) edge[red,very thick,->] (d1);
\end{tikzpicture}
$$
\caption{\label{fig:SGamma}
Two examples of graphs in $\langle\mS\Gamma^3\rangle$, with $\Gamma$ as in Figure \ref{fig:span}. The forest $F(a_1,a_2,a_3)$ is depicted red. The right example is also in $\langle\mO\Gamma^i\rangle$, while the left one is not.}
\end{figure}

The differential on $\bar\Sigma$ induces the differential $d_{E0}$ on $\langle\mS\Gamma^i\rangle$ and $\langle\mO\Gamma^i\rangle$, where the forbidden graphs are considered zero. It is straightforward to check that
\begin{equation}
\left(\langle\mS\Gamma^0\rangle,d_{EO}\right)=
\left(\langle\mS\Gamma\rangle,d_{E0}\right),
\qquad
\left(\langle\mO\Gamma^0\rangle,d_{EO}\right)=
\left(\langle\mO\Gamma\rangle,d_{E0}\right).
\end{equation}
Also, it holds that
\begin{equation}
\langle\mS\Gamma^{v-s}\rangle=\langle\mO\Gamma^{v-s}\rangle
\end{equation}
is one dimensional, spanned by the graph with edges $a_1,\dots,a_{v-s}$ of type $\ET$ and other edges of type $\Ess$.

There are natural projections of complexes $\pi_i:\langle\mS\Gamma^i\rangle\twoheadrightarrow\langle\mO\Gamma^0\rangle$ for every $i=0,\dots,v-s$.
Also for every $i=1,\dots,v-s$, there are natural maps
\begin{equation}
f^i:\langle\mO\Gamma^{i-1}\rangle\rightarrow\langle\mO\Gamma^i\rangle,\quad
g^i:\langle\mS\Gamma^{i-1}\rangle\rightarrow\langle\mS\Gamma^i\rangle,
\end{equation}
that only change the type of the edge $a_i$ as
\begin{equation}
\Ess\mapsto 0,\quad
\Ed\mapsto\ET,\quad
\dE\mapsto(-1)^{n+1}\ET,
\end{equation}
where forbidden graphs are considered zero.
The following diagram clearly commutes.
$$
\begin{tikzcd}
\langle\mO\Gamma^{i-1}\rangle
\arrow[twoheadleftarrow]{r}{\pi_{i-1}}
\arrow{d}{f^i}
& \langle\mS\Gamma^{i-1}\rangle
\arrow{d}{g^i} \\
\langle\mO\Gamma^{i}\rangle
\arrow[twoheadleftarrow]{r}{\pi_{i}}
& \langle\mS\Gamma^{i}\rangle
\end{tikzcd}
$$

\begin{lemma}
For every $i=1,\dots,v-s$ maps $f^i:\langle\mO\Gamma^{i-1}\rangle\rightarrow\langle\mO\Gamma^i\rangle$ and $g^i:\langle\mS\Gamma^{i-1}\rangle\rightarrow\langle\mS\Gamma^i\rangle$ are quasi-isomorphisms.
\end{lemma}
\begin{proof}
Let us prove the claim for $g^i$. 
The essential difference between $\langle\mS\Gamma^{i-1}\rangle$ and $\langle\mS\Gamma^{i}\rangle$ is in the edge $a_i$, it has to be of type $\ET$ in $\langle\mS\Gamma^{i}\rangle$, and it is of another type in $\langle\mS\Gamma^{i-1}\rangle$.
Since $g^i$ does not change types of other edges, it splits as a direct sum of maps between complexes with fixed types of other edges
$$
g^i_{fix}:\langle\mS\Gamma^{i-1}_{fix}\rangle\rightarrow\langle\mS\Gamma^i_{fix}\rangle
$$
where $\langle\mS\Gamma^{i-1}_{fix}\rangle$ and $\langle\mS\Gamma^{i}_{fix}\rangle$ are sub-complexes spanned by graphs with fixed types of all edges other than $a_i$. It is enough to show that each $g^i_{fix}$ is a quasi-isomorphism.

Here, depending on the choice of fixed edge types, the condition of being sourced can disallow some possibilities for the edge $a_i$ in both $\langle\mS\Gamma^{i-1}_{fix}\rangle$ and $\langle\mS\Gamma^{i}_{fix}\rangle$. We list all cases, showing that the map is a quasi-isomorphism in all of them. Let the vertex that is in the forest $F(a_1,\dots,a_{i})$ but not in the forest $F(a_1,\dots,a_{i-1})$ be called $x_i$.
\begin{enumerate}
\item If there is a vertex in the forest $F(a_1,\dots,a_{i-1})$ that has a neighboring edge of type $\Ed$ or $\dE$ heading towards it, or there is a vertex outside the forest $F(a_1,\dots,a_{i})$ that does not have a neighboring edge of type $\Ed$ or $\dE$ heading towards it, both $\langle\mS\Gamma^{i-1}_{fix}\rangle$ and $\langle\mS\Gamma^{i}_{fix}\rangle$ are zero complexes and the map is clearly quasi-isomorphism.
\item If it is not the case from (1) and $x_i$ has a neighboring edge of type $\Ed$ or $\dE$ heading towards it, the edge $a_i$ (that goes from a vertex in the forest $F(a_1,\dots,a_{i-1})$ towards $x_i$) can have types $\Ed$ or $\Ess$ in $\langle\mS\Gamma^{i-1}_{fix}\rangle$, making it acyclic. In $\langle\mS\Gamma^{i}_{fix}\rangle$, no type is allowed, so it is again the zero complex. Therefore, the map is again a quasi-isomorphism.
\item If it is not the case from (1) and $x_i$ does not have a neighboring edge of type $\Ed$ or $\dE$ heading towards it, the edge $a_i$ must have type $\Ed$ in $\langle\mS\Gamma^{i-1}_{fix}\rangle$. In $\langle\mS\Gamma^{i}_{fix}\rangle$ that edge must have type $\ET$, making the map an isomorphism. Thus, it is also a quasi-isomorphism.
\end{enumerate}

For the map $f^i$ between complexes of oriented graphs, one can easily verify that the extra condition does not change the argument.
\end{proof}

The lemma implies that
\begin{equation}
f:=f^{v-s}\circ\dots\circ f^1:
\langle\mO\Gamma\rangle\rightarrow\langle\mO\Gamma^{v-s}\rangle,
\end{equation}
and
\begin{equation}
g:=g^{v-s}\circ\dots\circ g^1:
\langle\mS\Gamma\rangle\rightarrow\langle\mS\Gamma^{v-s}\rangle
\end{equation}
are quasi-isomorphisms and the following diagram commutes:
$$
\begin{tikzcd}
\langle\mO\Gamma\rangle
\arrow[twoheadleftarrow]{r}{\pi}
\arrow{d}{f}
& \langle\mS\Gamma\rangle
\arrow{d}{g} \\
\langle\mO\Gamma^{v-s}\rangle
\arrow[leftrightarrow]{r}{\id}
& \langle\mS\Gamma^{v-s}\rangle.
\end{tikzcd}
$$

\begin{lemma}
\label{lem:QIfh}
The map $f\circ \Phi:\langle\Gamma\rangle\rightarrow\langle\mO\Gamma^{v-s}\rangle$ is a quasi-isomorphism.
\end{lemma}
\begin{proof}
Both complexes are one-dimensional, so we only need to prove that $f\circ \Phi\neq 0$.
The left-hand side complex has a generator $\Gamma$. It holds that
$$
f\circ \Phi(\Gamma)=
f\left(\sum_{\tau\in F(\Gamma)}\Phi_{\tau}(\Gamma)\right).
$$
The map $\Phi_{\tau}$ gives edges in $E(\tau)$ type $\Ed$ or $\dE$, and type $\Ess$ to the other edges. After that, the map $f=f^1\circ\dots\circ f^{v-s}$ kills all graphs with any of edges $a_1,\dots,a_{v-s}$ being of type $\Ess$. Therefore, $f\circ h_{x,\tau}$ is non-zero only if the forest $\tau$ consist exactly of the edges $a_1,\dots,a_{v-s}$. Let us call this forest $T$. So
\begin{equation}
f\circ \Phi(\Gamma)=
f\left(\Phi_{T}(\Gamma)\right).
\end{equation}
It is clearly the generator of $\langle\mO\Gamma^{v-s}\rangle$, and therefore non-zero.
\end{proof}

With this lemma, we have proven that the diagonal map in the following commutative diagram is also a quasi-isomorphism.
$$
\begin{tikzcd}
\langle\Gamma\rangle
\arrow{r}{\Phi}
\arrow[swap]{dr}{f\circ \Phi}
& \langle\mO\Gamma\rangle
\arrow[twoheadleftarrow]{r}{\pi}
\arrow{d}{f}
& \langle\mS\Gamma\rangle
\arrow{d}{g} \\
& \langle\mO\Gamma^{v-s}\rangle
\arrow[leftrightarrow]{r}{\id}
& \langle\mS\Gamma^{v-s}\rangle
\end{tikzcd}
$$
Together with the previous result we conclude that all mentioned maps are quasi-isomorphisms. Which concludes the proof.
\end{proof}

\begin{cor}
The map 
$$
\Phi:  \left(\HG_n,d+h\right)\to\left( \mO \G_{n+1},d\right)
$$
and the projection
$$
p:\left(\mS\G_n,d\right)\twoheadrightarrow\left(\mO\G_n,d\right)
$$
are quasi-isomorphism.
\end{cor}
\begin{proof}
On the mapping cone of $\Phi$, we set up a spectral sequence on the number $s$, which is the number of hairs in $\HG_n$ and number of sources in $\mO \G_{n+1}$. 
Standard splitting of complexes as the product of complexes with fixed loop number implies the correct convergence.
On the first page of the spectral sequence we have exactly the mapping cone of
$$
\Phi:  \left(\HG_n,d\right)\to\left( \mO \G_{n+1},d_0\right)
$$
which is acyclic according to Theorem \ref{thm:main1}. By the spectral sequence argument, the mapping cone of
$$
\Phi:  \left(\HG_n,d+h\right)\to\left( \mO \G_{n+1},d\right)
$$
is also acyclic, hence $\Phi$ is here a quasi-isomorphism.

The same argument works for the projection $p$. However, it has already been proven in \cite[Theorem 1.1]{MultiSourced}.
\end{proof}

\begin{cor}[Theorem \ref{thm:main}]
The  map dual map of $G\leftrightarrow \Phi$, given explicitly in \eqref{eq:G}, is a quasi-isomorphism of complexes
$$
G:  \left( \mO \GC_{n+1},\delta \right)\to \left(\HGC_n,\delta+\chi\right)
$$
and of complexes
$$
G:  \left( \mO \GC_{n+1},\delta_0\right)\to \left(\HGC_n,\delta\right).
$$
\end{cor}
\begin{proof}
We may fix the loop order $b$.
As $\mB_b\mO \G_{n+1}$ and $\mB_b\HG_{n}$ are finite dimensional in each degree, it is clear that each quasi-isomorphism
$$
F: (\mB_b\mO \G_{n+1}, d )\to (\mB_b\HG_{n},d+h),\quad
F: (\mB_b\mO \G_{n+1}, d_0 )\to (\mB_b\HG_{n},d)$$
has a dual quasi-isomorphism
$$
G: (\mB_b\mO \GC_{n}, \delta )\to (\mB_b\HGC_{n},\delta+\chi),\quad
G: (\mB_b\mO \GC_{n}, \delta )\to (\mB_b\HGC_{n+1},\delta_0).$$
As $\Phi$ preserves the loop order, its dual map $G$ must be a quasi-isomorphism.
\end{proof}

\section{Application to ribbon graphs and the moduli space of curves with punctures}
\label{s:rib}

In this section, we will follow \cite{MW2} in order to define the ribbon graph complex $(\RGC[1],\delta+ \Delta_1)$, as well as the map $F:(\mO \GC_1,\delta) \to (\RGC[1],\delta+\Delta_1)$. This will allow us to make the observation that $F$ is also a map of complexes
$$F:(\mO \GC_1,\delta_0)\to (\RGC[1],\delta).$$

\subsection{Ribbon Graphs}
\begin{defi}
A \emph{ribbon graph} $\Gamma$ is a triple $(F_\Gamma, \iota_{\Gamma}, \sigma_{\Gamma})$, where $F_{\Gamma}$ is a finite set, $\iota_{\gamma}:F_{\Gamma}\to F_{\Gamma}$ is an involution with no fixed points, i.e.\
$$\iota_\Gamma^2 = id,\quad \iota_\Gamma(f)\neq f \quad ,$$
for every $f\in F_\Gamma$, and $\sigma_{\Gamma}:F_{\Gamma}\to F_{\Gamma}$ is a bijection.
\end{defi}

The elements of $F_{\Gamma}$ are called \emph{flags} or \emph{half edges}. The orbits of the involution $\iota_{\Gamma}$ are called \emph{edges}, and the set of all edges will be denoted by $E(\Gamma)$. The orbits of $\sigma_{\Gamma}$ are called \emph{vertices}, and the set of all vertices will be denoted by $V(\Gamma)$.

\begin{defi}
We say that a cyclic ordering on a finite set $A$ is a $\mathbb{Z}_{|A|}$ action on $A$ with precisely $1$ orbit.
\end{defi}
The difference between an ordinary graph and a ribbon graph is that each vertex in a ribbon graph is equipped with a cyclic ordering of its adjacent (half) edges, given by $f+1= \sigma_\Gamma f.$  

We may draw a picture of a ribbon graph $\Gamma$ in the following way: 
\begin{enumerate}
\item For each vertex $(i_1 i_2\ldots i_k)\in V(\Gamma)$, draw a dot with clockwise ordered lines labeled by $i_1,i_2,\ldots i_k$ connected to the dot
$$(i_1 i_2\ldots i_k)\leftrightarrow
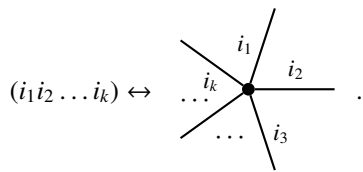
\begin{tikzpicture}[baseline=-3.5,>=stealth',shorten >=1pt,auto,node distance=1.5cm,
                    thick,main node/.style={circle,draw,font=\sffamily\bfseries}]
  \node[main node,int] (1) {};
  \node[] (2) at ({360/5 }:1.3cm) {};
  \node[] (3) at ({2*360/5 }:1.3cm) {};

   \node[] (0) at ({3.5*360/5 }:0.7cm) {$\cdots$};  
   
   \node[] (0) at ({2.7*360/5 }:0.7cm) {$\cdots$};  
     \node[] (5) at ({4*360/5 }:1.3cm) {};   
     
          \node[] (4) at ({3*360/5 }:1.3cm) {}; 
  \node[] (6) at ({5*360/5 }:1.3cm) {};   
  
  \path[every node/.style={font=\sffamily\small}](1) edge node[left]{} (4);  
  \path[every node/.style={font=\sffamily\small}](1) edge node[left]{$i_1$} (2);
 \path[every node/.style={font=\sffamily\small}](1) edge node[below]{$i_k$} (3);
  \path[every node/.style={font=\sffamily\small}](1) edge node[right]{$i_3$}(5);
   \path[every node/.style={font=\sffamily\small}](1) edge node[above]{$i_2$}(6);

\end{tikzpicture}.$$

\item For each edge $(i_a i_b)\in E(\Gamma)$, connect the lines labeled by $i_a$ and $i_b$.

\end{enumerate}
Two very basic examples of ribbon graphs are
$$(\{1,2\}, (12), (1)(2))= 
\begin{tikzpicture}[baseline=-.55ex,scale=1]
 \node[int] (a) at (0,0) {};
 \node[int] (b) at (1,0) {};
 \draw (a) -- (b) node[near start, above]{$_1$} node[near end, above]{$_2$};
\end{tikzpicture}  \quad \quad(\{1,2\}, (12), (12))=
\begin{tikzpicture}[every loop/.style={},baseline=-.55ex,scale=1]
 \node[int] (A) at (0,0) {};
  \draw (a) edge[loop]  node[near start, above]{$_2$}node[near end, above]{$_1$}(a); 	
\end{tikzpicture}.$$

We call the orbits of the permutation $\sigma^{-1}_{\Gamma} \circ \iota_{\Gamma}$ \emph{boundaries} of $\Gamma$, and we denote the set of boundaries by $B(\Gamma)$. 
For example, the ribbon graph
$$
\begin{tikzpicture}[baseline=-.55ex,scale=1]
 \node[int] (a) at (0,0) {};
 \node[int] (b) at (1,0) {};
 \draw (a) -- (b) node[near start, above]{$_1$} node[near end, above]{$_2$};
 \end{tikzpicture}$$
has one boundary $(12)$, while the ribbon graph 
$$\begin{tikzpicture}[every loop/.style={},baseline=-.55ex,scale=1]
 \node[int] (A) at (0,0) {};
  \draw (a) edge[loop]  node[near start, above]{$_2$}node[near end, above]{$_1$} (a); 	
\end{tikzpicture}$$
has two boundaries, $(1)$ and $(2)$.

\subsection{A PROP of ribbon graphs}
Let $\rgra_{n,m,k}$ be the set of ribbon graphs with vertex set labeled by $[n]$, boundary set labeled by $[m]$ and edge set labeled by $[k]$. That is ribbon graphs 
$$\Gamma=([k]\sqcup[k],\iota_k, \sigma_{\Gamma}),$$
where $\iota_k$ is the natural involution on $[k]\sqcup [k]$, together with bijections
$v_{\Gamma}:[n]\to V(\Gamma), \quad b_{\Gamma}:[m]\to B(\Gamma)$. Here, the edge labeled by $i\in [k]$ is the pair $(i_1,i_2)$, $i=i_1\in [k]\sqcup \emptyset$, $i=i_2\in \emptyset \sqcup [k]$. We say that the edge $(i_1,i_2)$ is \emph{intrinsically oriented} from $i_1$ to $i_2$. 

The group $\mathbb{P}_k:=\sym_{k}\ltimes \sym^{\times k}_2$ acts naturally on $\rgra_{n,m,k}$ by permuting edges and reversing the intrinsic orientation. Let $\RGra_d(n,m)$ be the vector space
$$\RGra(n,m):= \prod_{k\ge 0}\left(\Bbbk\langle \rgra_{n,m,k}\rangle [k] \otimes \sgn_k\right)^{\mathbb{P}_k} 
$$
The space 
$$\RGra:=\bigoplus_{n,m\ge 1} \RGra(n,m) $$ 
is an $\sym$-bimodule, where $\sym_n$ acts on $\RGra(n,m)$ by permuting vertex labels, and $\sym_m$ acts on $\RGra(n,m)$ by permuting boundary labels.

In order to define the properadic composition maps
$\circ :\RGra\otimes \RGra\to \RGra $, we have to make a few definitions.
\begin{defi}
Let $A$ and $B$ be two finite sets with cyclic orderings. We say that an \emph{ordered $A$-partition $p$ of $B$} is a partition
$$\bigsqcup_{a\in A} p_a =B,$$
where each
$$p_a=\{ b_a^1,\ldots, b_a^{k}\}\subset B$$
is ordered with 
$$b_a^i+1= \begin{cases} b_a^{i+1} & i<k\\
 b_{a+r}^1 & i=k, \text{ where $r=\min\{j\in\Z|j\geq 1, p_{a+j}\neq \emptyset\}$.}
\end{cases}.$$
We denote the set of all ordered $A$-partitions of $B$ by $P(A,B)$.
\end{defi}

\begin{defi}
Let $\Gamma=(F,\iota, \sigma)$ be a ribbon graph, and let $v\in V(\Gamma)$, $b\in B(\Gamma)$ such that $v\cap b=\emptyset.$ Then, for each ordered partition $p\in P(b,v)$, we define the \emph{$p$-grafted} ribbon graph 
$$\circ_p \Gamma= (F,\iota,\circ_p \sigma ),$$
by letting $\circ_p \sigma$ be the unique permutation such that
\begin{enumerate}
\item $\circ_p \sigma |_{F \setminus (v\cup b)} = \sigma|_{F \setminus (v\cup v)}$; 
\item for each $j\in b$
$$\circ_p \sigma(j) = \begin{cases} \min p_j & p_j \neq \emptyset,\\
\sigma(j) &  p_j = \emptyset;
\end{cases}$$
\item for each $i\in p_j$
$$\circ_p \sigma(i) = \begin{cases}\sigma(i) & i \neq \max {p_j},\\
\sigma(j) &  i = \max {p_j}.
\end{cases}$$
\end{enumerate}
\end{defi}

For two ribbon graphs $\Gamma_1=(F_1,\iota_1,\sigma_1)$ and $\Gamma_2=(F_2,\iota_2,\sigma_2)$, $v\in V(\Gamma_1)$, $b\in B(\Gamma_2)$, and  $p\in P(b,v)$, we set
$$
\Gamma_1\circ_p \Gamma_2 := \left(F_1\sqcup F_2,\iota_1\sqcup\iota_2,\circ_p (\sigma_1\sqcup \sigma_2)\right).
$$

A picture of the ribbon graph $\circ_p \Gamma$ is obtained from a picture of the ribbon graph $\Gamma$ by removing the vertex $v$ and reconnecting its adjacent edges to the corners of the boundary $b$ according to the partition $p$.

\begin{lemma} \label{lemma:ribbon_PROP}
There is a bijection of vertices
$$p_V: V(\Gamma)\setminus\{v\} \to V(\circ_p \Gamma)$$
and a bijection of boundaries
$$p_B: B(\Gamma)\setminus\{b\} \to B(\circ_p \Gamma),$$
such that $v'\subseteq p_V(v')$ and $b'\subseteq p_B(b')$ for every $v'\in V(\Gamma)\setminus\{v\}$ and $b'\in B(\Gamma)\setminus\{b\}$.

Furthermore,  if $v'\cap b=\emptyset$ we have equality $v'= p_V(v')$. Similarly if $v\cap b'=\emptyset$, we have $b'= p_B(b')$. 
\end{lemma}
\begin{proof}
Pick a boundary $b'\in B(\Gamma)\setminus \{b\},$ and a flag $j'\in b'$. Suppose that $(\circ_p \sigma)^{-1}\iota(j')\neq \sigma^{-1}\iota(j')$. Then we must have $\iota(j')=\sigma(j)$ or $\iota(j')= \min p_j$ for some $j\in b$. The first case implies that $\sigma^{-1} \iota(j')\in b$ and, therefore, we get $j'\in b$, which contradicts our choice of $j'$. If $\iota(j')=\min p_j$, then $(\circ_p \sigma)^{-1}\iota(j')\in b$. Furthermore, for $r=1,2,3,\ldots$, we have that $((\circ_p \sigma)^{-1}\iota)^r(j')=(\sigma^{-1}\iota)^r(j') \in b$ until $\iota(\sigma^{-1}\iota)^{r-1}(j')=\sigma(j)$, in which case  $\iota(\circ_p\sigma^{-1}\iota)^{r}(j')= \sigma^{-1}\iota (j').$ It follows that $b'$ is a subset of the orbit of $\iota(\circ_p\sigma^{-1}\iota)^{r}(j')$. Thus, we may define the map 
$$p_B: B(\Gamma)\setminus\{v\} \to B(\circ_p \Gamma),$$
by setting $p_2(b')$ to be the $(\circ_p\sigma)^{-1}\iota$ orbit of any $j'\in b'$. From the arguments above, it follows that $p_2$ is injective and $b'\subset p_2(b')$.
Finally, we note that for any $j\in b$, there must exist an $r\ge 1$ such that $\iota(\circ_p\sigma^{-1}\iota)^{r}(j)\notin b$. Hence $p_2$ must also be surjective. 

If $v\cap b'=\emptyset$, then it is clear from the construction that $b'= p_B(b')$.

Similarly, we can define 
$$p_V: V(\Gamma)\setminus\{v\} \to V(\circ_p \Gamma)$$
by setting  $p_V(v')$ to be the $\circ_p\sigma$ orbit of any $i\in v'$.
By the same arguments as above, we get that $p_V$ is well defined, bijective, and $v'\subset p_V(v')$. 
\end{proof}
This lemma implies that, for mutually disjoint $v_1,v_2\in V(\Gamma)$, $b_1,b_2\in B(\Gamma)$, and partitions $p_1\in P(b_1,v_1)$, $p_2\in P(b_2,v_2)$, we may define
$$\circ_{p_1,p_2} \Gamma:= \circ_{p_2}(\circ_{p_1} \Gamma).$$
It is clear that we have
\begin{equation} \label{eq:ribbon_comp_ass}
\circ_{p_1,p_2} \Gamma = \circ_{p_2}(\circ_{p_1} \Gamma)= \circ_{p_1}(\circ_{p_2} \Gamma)=\circ_{p_2,p_1} \Gamma.
\end{equation}

For each $k\le m_1,n_2$, we define the properadic composition maps
\begin{eqnarray*}
\circ_k :\RGra_d(n_1,m_1)\otimes \RGra_d(n_2,m_2)&\to &\RGra_d(n_1+n_2-k,m_1+m_2-k),
\end{eqnarray*}

$$\begin{tikzpicture}[baseline=-0.55ex,scale=0.6]
  \node[] (0) at (0,0) {$_{\Gamma_1}$};
  \node[] (1) at (-1,1.2) {$_1$};
  \node[] (2) at (-0.6,1.2) {$_2$};
  \node[] (.1) at (0.2,0.8) {\ldots};
  \node[] (4) at (1,1.2) {$_{m_1}$};
 
   \draw (0) edge[->] node {} (1);   
   \draw (0) edge[->] node {} (2);   
   \draw (0) edge[->] node {} (4);

     \node[] (-1) at (-1,-1.2) {$_1$};
  \node[] (-2) at (-0.6,-1.2) {$_2$};
  \node[] (-.1) at (0.2,-.9) {\ldots};
  \node[] (-4) at (1,-1.2) {$_{n_1}$};

   \draw (0) edge[<-] node {} (-1);   
   \draw (0) edge[<-] node {} (-2);   
   \draw (0) edge[<-] node {} (-4);  
\end{tikzpicture} \otimes \begin{tikzpicture}[baseline=-0.55ex,scale=0.6]
  \node[] (0) at (0,0) {$_{\Gamma_2}$};
  \node[] (1) at (-1,1.2) {$_1$};
  \node[] (2) at (-0.6,1.2) {$_2$};
  \node[] (.1) at (0.2,0.8) {\ldots};
  \node[] (4) at (1,1.2) {$_{m_2}$};
 
   \draw (0) edge[->] node {} (1);   
   \draw (0) edge[->] node {} (2);   
   \draw (0) edge[->] node {} (4);

     \node[] (-1) at (-1,-1.2) {$_1$};
  \node[] (-2) at (-0.6,-1.2) {$_2$};
  \node[] (-.1) at (0.2,-.9) {\ldots};
  \node[] (-4) at (1,-1.2) {$_{n_2}$};

   \draw (0) edge[<-] node {} (-1);   
   \draw (0) edge[<-] node {} (-2);   
   \draw (0) edge[<-] node {} (-4);  
\end{tikzpicture} 
\mapsto
\begin{tikzpicture}[baseline=-0.55ex,scale=0.6]
  \node[] (0) at (0,0) {$_{\Gamma_1}$};
 \node[] (ddots) at (1,0.5) {$\tiny{\ddots}$};

  \node[] (1) at (-1,1.5) {};
  \node[] (2) at (-0.6,1.5) {};
  \node[] (.1) at (0.2,1.1) {$\ldots$};
  \node[] (4) at (1,1.5) {};
  
\node[] (.1) at (0,1.7) {$\overbrace{\qquad \quad }^{1,\ldots, m_1-k}$};

  \node[] (10) at (2,1) {$_{\Gamma_2}$};
    \node[] (11) at (1.1,-.6) {};
  \node[] (12) at (1.5,-0.6) {};
  \node[] (va) at (2.2,-0.2) {$\ldots$};
  \node[] (13) at (3,-0.6) {};  
  
  \node[] (.1) at (2,-0.9) {$\underbrace{\qquad \quad }_{ n_2-k+1,\ldots, n_2}$};

  
   \draw (10) edge[<-] node {} (11);   
   \draw (10) edge[<-] node {} (12);
     \draw (10) edge[<-] node {} (13);
  
  \node[] (21) at (1.1,2.1) {};
  \node[] (22) at (1.5,2.1) {};
    \node[] (-.1) at (2.2,1.7) {$\ldots$};
  \node[] (23) at (3,2.1) {};

\node[] (.1) at (2,2.3) {$\overbrace{\qquad \quad }^{ m_2-k+1, \ldots, m_2 }$};

     \draw (10) edge[->] node {} (21);   
   \draw (10) edge[->] node {} (22);
     \draw (10) edge[->] node {} (23);

    \draw (0) edge[->,bend left] node {} (10);  

    \draw (0) edge[->,bend left=10] node {} (10);  
    \draw (0) edge[->,bend right] node {} (10);  

   \draw (0) edge[->] node {} (1);   
   \draw (0) edge[->] node {} (2);   
   \draw (0) edge[->] node {} (4);

     \node[] (-1) at (-1,-1) {};
  \node[] (-2) at (-0.6,-1) {};
  \node[] (-.1) at (0.2,-.7) {\ldots};
  \node[] (-4) at (1,-1) {};

\node[] (.1) at (0,-1.2) {$\underbrace{\qquad  \quad}_{1,\ldots, n_1-k}$};  
   \draw (0) edge[<-] node {} (-1);   
   \draw (0) edge[<-] node {} (-2);   
   \draw (0) edge[<-] node {} (-4);  
\end{tikzpicture} \mapsto 
\begin{tikzpicture}[baseline=-0.55ex,scale=0.6]
  \node[] (0) at (0,0) {$_{\Gamma_1\circ_k\Gamma_2}$};
  \node[] (1) at (-1,1.2) {$_1$};
  \node[] (2) at (-0.6,1.2) {$_2$};
  \node[] (.1) at (0.2,0.8) {\ldots};
  \node[] (4) at (1,1.2) {$_{m_1+m_2-k}$};
 
   \draw (0) edge[->] node {} (1);   
   \draw (0) edge[->] node {} (2);   
   \draw (0) edge[->] node {} (4);

     \node[] (-1) at (-1,-1.2) {$_1$};
  \node[] (-2) at (-0.6,-1.2) {$_2$};
  \node[] (-.1) at (0.2,-.9) {\ldots};
  \node[] (-4) at (1,-1.2) {$_{n_1+n_2-k}$};

   \draw (0) edge[<-] node {} (-1);   
   \draw (0) edge[<-] node {} (-2);   
   \draw (0) edge[<-] node {} (-4);  
\end{tikzpicture} 
$$

composing the boundaries $m_1-k+1,\ldots, m_1$ of $\Gamma_1$ to the vertices $1,\ldots, k$ of $\Gamma_2$, by
$$\Gamma_1\circ_{k} \Gamma_2:= \prod_{i=1}^{k}\left(\sum_{p \in P(b_{\Gamma_1}(n_1-k+i),v_{\Gamma_2}(i))} \circ_{p}(\Gamma_1\sqcup \Gamma_2)\right).$$

The element $\Gamma_1\circ_k \Gamma_2$ is the sum of all graphs obtained from $\Gamma_1$ and $\Gamma_2$ by
\begin{enumerate}
\item Remove each vertex $1,\ldots, k $ from $\Gamma_2$;
\item For each $i\in [k],$ reconnecting each half edge in $v_{\Gamma_2}(i)$ to a corner of the boundary $b_{\Gamma_1}(m_1-k+i)$, respecting the cyclic orientations.
\end{enumerate}

\begin{prop}
The composition maps $\circ_{\bu}:\RGra\otimes \RGra \to \RGra$ defines a PROP structure on $\RGra$.
\end{prop}
\begin{proof}It follows by Lemma \ref{lemma:ribbon_PROP} and \eqref{eq:ribbon_comp_ass} that the maps are well defined and well behaved.
\end{proof}

Let $\LieB_{0,0}$ be the PROP of (degree shifted) Lie bi-algebras, i.e. the PROP generated by symmetric $3$-valent corollas of degree $1$ 
\begin{equation}\label{eq:Lie_Corollas} \begin{tikzpicture}[baseline=-.55ex,scale=0.6]
 \node[int] (a) at (0,0) {};
  \node[](out1) at (0.7,-0.7) {$_2$};
   \node[](out2) at (-.7,-0.7) {$_1$};
   \node[](in) at (0,.7) {};
     \draw (out1) edge[->] (a);
      \draw (out2) edge[->] (a);
       \draw (a) edge[->] (in);
\end{tikzpicture}=  \begin{tikzpicture}[baseline=-.55ex,scale=0.6]
 \node[int] (a) at (0,0) {};
  \node[](out1) at (0.7,-0.7) {$_1$};
   \node[](out2) at (-.7,-0.7) {$_2$};
   \node[](in) at (0,.7) {};
     \draw (out1) edge[->] (a);
      \draw (out2) edge[->] (a);
       \draw (a) edge[->] (in);
\end{tikzpicture},
 \begin{tikzpicture}[baseline=-.55ex,scale=0.6]
 \node[int] (a) at (0,0) {};
  \node[](out1) at (0.7,0.7) {$_2$};
   \node[](out2) at (-.7,0.7) {$_1$};
   \node[](in) at (0,-.7) {};
     \draw (a) edge[->] (out1);
      \draw (a) edge[->] (out2);
       \draw (in) edge[->] (a);
   
\end{tikzpicture}= \begin{tikzpicture}[baseline=-.55ex,scale=0.6]
 \node[int] (a) at (0,0) {};
  \node[](out1) at (0.7,0.7) {$_1$};
   \node[](out2) at (-.7,0.7) {$_2$};
   \node[](in) at (0,-.7) {};
     \draw (a) edge[->] (out1);
      \draw (a) edge[->] (out2);
       \draw (in) edge[->] (a);
   
\end{tikzpicture},
\end{equation}
modulo the relations

\begin{equation} \label{eq:jacobi}
\begin{tikzpicture}[baseline=-.55ex,scale=0.6]
 \node[int] (a) at (0,0) {};
  \node[int](out1) at (0.7,-0.7) {};
   \node[](out2) at (-.7,-0.7) {$_1$};
	  \node[](out3) at (1.4,-1.4) {$_3$};
   \node[](out4) at (0,-1.4) {$_2$};   
   
   \node[](in) at (0,.7) {};
     \draw (a) edge[<-] (out1);
      \draw (a) edge[<-] (out2);
           \draw (out1) edge[<-] (out3);
      \draw (out1) edge[<-] (out4);
       \draw (in) edge[<-] (a);
\end{tikzpicture}+\begin{tikzpicture}[baseline=-.55ex,scale=0.6]
 \node[int] (a) at (0,0) {};
  \node[int](out1) at (0.7,-0.7) {};
   \node[](out2) at (-.7,-0.7) {$_2$};
	  \node[](out3) at (1.4,-1.4) {$_1$};
   \node[](out4) at (0,-1.4) {$_3$};   
   
   \node[](in) at (0,.7) {};
     \draw (a) edge[<-] (out1);
      \draw (a) edge[<-] (out2);
           \draw (out1) edge[<-] (out3);
      \draw (out1) edge[<-] (out4);
       \draw (in) edge[<-] (a);
\end{tikzpicture}+\begin{tikzpicture}[baseline=-.55ex,scale=0.6]
 \node[int] (a) at (0,0) {};
  \node[int](out1) at (0.7,-0.7) {};
   \node[](out2) at (-.7,-0.7) {$_3$};
	  \node[](out3) at (1.4,-1.4) {$_2$};
   \node[](out4) at (0,-1.4) {$_1$};   
   
   \node[](in) at (0,.7) {};
     \draw (a) edge[<-] (out1);
      \draw (a) edge[<-] (out2);
           \draw (out1) edge[<-] (out3);
      \draw (out1) edge[<-] (out4);
       \draw (in) edge[<-] (a);
\end{tikzpicture}=0,
\end{equation}

\begin{equation} \label{eq:cojac}
\begin{tikzpicture}[baseline=-.55ex,scale=0.6]
 \node[int] (a) at (0,0) {};
  \node[int](out1) at (0.7,0.7) {};
   \node[](out2) at (-.7,0.7) {$_1$};
	  \node[](out3) at (1.4,1.4) {$_3$};
   \node[](out4) at (0,1.4) {$_2$};   
   
   \node[](in) at (0,-.7) {};
     \draw (a) edge[->] (out1);
      \draw (a) edge[->] (out2);
           \draw (out1) edge[->] (out3);
      \draw (out1) edge[->] (out4);
       \draw (in) edge[->] (a);
\end{tikzpicture}+\begin{tikzpicture}[baseline=-.55ex,scale=0.6]
 \node[int] (a) at (0,0) {};
  \node[int](out1) at (0.7,0.7) {};
   \node[](out2) at (-.7,0.7) {$_2$};
	  \node[](out3) at (1.4,1.4) {$_1$};
   \node[](out4) at (0,1.4) {$_3$};   
   
   \node[](in) at (0,-.7) {};
     \draw (a) edge[->] (out1);
      \draw (a) edge[->] (out2);
           \draw (out1) edge[->] (out3);
      \draw (out1) edge[->] (out4);
       \draw (in) edge[->] (a);
\end{tikzpicture}+\begin{tikzpicture}[baseline=-.55ex,scale=0.6]
 \node[int] (a) at (0,0) {};
  \node[int](out1) at (0.7,0.7) {};
   \node[](out2) at (-.7,0.7) {$_3$};
	  \node[](out3) at (1.4,1.4) {$_2$};
   \node[](out4) at (0,1.4) {$_1$};   
   
   \node[](in) at (0,-.7) {};
     \draw (a) edge[->] (out1);
      \draw (a) edge[->] (out2);
           \draw (out1) edge[->] (out3);
      \draw (out1) edge[->] (out4);
       \draw (in) edge[->] (a);
\end{tikzpicture}=0,
\end{equation}

\begin{equation} \label{eq:IHXor}
\begin{tikzpicture}[baseline=-.55ex,scale=0.6]
 \node[int] (a) at (0,0) {};
  \node[int](out1) at (0,0.7) {};

   \node[](out4) at (-.7,1.4) {$_1$};   

 \node[](out3) at (.7,1.4) {$_2$};

      \node[](out2) at (-.7,-0.7) {$_1$};
   \node[](in) at (.7,-.7) {$_2$};
     \draw (a) edge[->] (out1);
      \draw (a) edge[<-] (out2);
           \draw (out1) edge[->] (out3);
      \draw (out1) edge[->] (out4);
       \draw (in) edge[->] (a);
\end{tikzpicture}
+
\begin{tikzpicture}[baseline=-.55ex,scale=0.6]
 \node[int] (a) at (0,0) {};
  \node[int](out1) at (0.7,0.7) {};

   \node[](out4) at (-.7,1.4) {$_1$};   

 \node[](out3) at (.7,1.4) {$_2$};

      \node[](out2) at (0,-.7) {$_1$};
   \node[](in) at (1.4,-.7) {$_2$};

     \draw (a) edge[->] (out1);
      \draw (a) edge[<-] (out2);
           \draw (out1) edge[->] (out3);
      \draw (a) edge[->] (out4);
       \draw (in) edge[->] (out1);
\end{tikzpicture}+
\begin{tikzpicture}[baseline=-.55ex,scale=0.6]
 \node[int] (a) at (0,0.7) {};
  \node[int](out1) at (0.7,0) {};

   \node[](out4) at (0,1.4) {$_1$};   

 \node[](out3) at (1.4,1.4) {$_2$};

      \node[](out2) at (-.7,-.7) {$_1$};
   \node[](in) at (0.7,-.7) {$_2$};

     \draw (a) edge[<-] (out1);
      \draw (a) edge[<-] (out2);
           \draw (out1) edge[->] (out3);
      \draw (a) edge[->] (out4);
       \draw (in) edge[->] (out1);
\end{tikzpicture}
+
\begin{tikzpicture}[baseline=-.55ex,scale=0.6]
 \node[int] (a) at (0,0) {};
  \node[int](out1) at (0.7,0.7) {};

   \node[](out4) at (-.7,1.4) {$_2$};   

 \node[](out3) at (.7,1.4) {$_1$};

      \node[](out2) at (0,-.7) {$_1$};
   \node[](in) at (1.4,-.7) {$_2$};

     \draw (a) edge[->] (out1);
      \draw (a) edge[<-] (out2);
           \draw (out1) edge[->] (out3);
      \draw (a) edge[->] (out4);
       \draw (in) edge[->] (out1);
\end{tikzpicture}+
\begin{tikzpicture}[baseline=-.55ex,scale=0.6]
 \node[int] (a) at (0,0.7) {};
  \node[int](out1) at (0.7,0) {};

   \node[](out4) at (0,1.4) {$_1$};   

 \node[](out3) at (1.4,1.4) {$_2$};

      \node[](out2) at (-.7,-.7) {$_2$};
   \node[](in) at (0.7,-.7) {$_1$};

     \draw (a) edge[<-] (out1);
      \draw (a) edge[<-] (out2);
           \draw (out1) edge[->] (out3);
      \draw (a) edge[->] (out4);
       \draw (in) edge[->] (out1);
\end{tikzpicture}=0.
\end{equation}

\begin{prop} [\cite{MW}]\label{prop:LieB_{0,0}-RGra} 
There is a map of PROPs
$$s:\LieB_{0,0} \to \RGra$$
given by
$$
 \begin{tikzpicture}[baseline=-.55ex,scale=0.6]
 \node[int] (a) at (0,0) {};
  \node[](out1) at (0.7,-0.7) {$_2$};
   \node[](out2) at (-.7,-0.7) {$_1$};
   \node[](in) at (0,.7) {$_1$};
     \draw (out1) edge[->] (a);
      \draw (out2) edge[->] (a);
       \draw (a) edge[->] (in);
\end{tikzpicture} \mapsto \begin{tikzpicture}[baseline=-.55ex,scale=0.6]
 \node[draw,circle] (a) at (0,0) {$_1$};
 \node[draw,circle] (b) at (1.5,0) {$_2$};
 \node[] (c) at (0.75,0.5) {$_1$}; 
 
 \draw (a) edge[->] (b);
\end{tikzpicture}, \quad\quad 
 \begin{tikzpicture}[baseline=-.55ex,scale=0.6]
 \node[int] (a) at (0,0) {};
  \node[](out1) at (0.7,0.7) {$_2$};
   \node[](out2) at (-.7,0.7) {$_1$};
   \node[](in) at (0,-.7) {$_1$};
     \draw (a) edge[->] (out1);
      \draw (a) edge[->] (out2);
       \draw (in) edge[->] (a);
   
\end{tikzpicture} \mapsto \begin{tikzpicture}[baseline=-.55ex,scale=0.6]
 \node[draw,circle] (a) at (0,0) {$_1$};
  \draw (a) edge[loop]  node[midway, below]{$_{1}$} node[midway, above]{$_{2}$}(a); 	
\end{tikzpicture},
$$
\end{prop}
\begin{rem}
The map $s$ factors through the PROP of involutive Lie bialgebras $\LieB_{0,0}^{\diamond}$ which is generated by the corollas \eqref{eq:Lie_Corollas} modulo the relations \eqref{eq:jacobi},\eqref{eq:cojac},  \eqref{eq:IHXor} plus the additional relation
\begin{equation}\label{eq:involutive}\begin{tikzpicture}[baseline=-.55ex,scale=0.6]
 \node[int] (a) at (0,-0.3) {};
  \node[int](out1) at (0,0.4) {};

   \node[](out4) at (0,1) {};

   \node[](in) at (0,-.9) {};

     \draw (a) edge[->,bend left] (out1);
     \draw (a) edge[->,bend right] (out1);

     \draw (in) edge[->] (a);
     \draw (out1) edge[->] (out4);
\end{tikzpicture} =0.
\end{equation}
\end{rem}
\subsection{The ribbon graph complex $\RGC$}

Let $\hoLieB_{0,0}$ be the quasi-free differential graded PROP generated by skew symmetric corollas
$$\begin{tikzpicture}[baseline=-0.55ex,scale=0.6]
  \node[int] (0) at (0,0) {};
  \node[] (1) at (-1,1.2) {$_1$};
  \node[] (2) at (-0.6,1.2) {$_2$};
  \node[] (.1) at (0.2,0.8) {\ldots};
  \node[] (4) at (1,1.2) {$_{m}$};
 
   \draw (0) edge[->] node {} (1);   
   \draw (0) edge[->] node {} (2);   
   \draw (0) edge[->] node {} (4);

     \node[] (-1) at (-1,-1.2) {$_1$};
  \node[] (-2) at (-0.6,-1.2) {$_2$};
  \node[] (-.1) at (0.2,-.9) {\ldots};
  \node[] (-4) at (1,-1.2) {$_{n}$};

   \draw (0) edge[<-] node {} (-1);   
   \draw (0) edge[<-] node {} (-2);   
   \draw (0) edge[<-] node {} (-4);  
\end{tikzpicture} \quad n\ge 1, m\ge 1, n+m\ge 3$$
of degree $1$, with the differential
 $$\delta \begin{tikzpicture}[baseline=-0.55ex,scale=0.6]
  \node[int] (0) at (0,0) {};
  \node[] (1) at (-1,1.2) {$_1$};
  \node[] (2) at (-0.6,1.2) {$_2$};
  \node[] (.1) at (0.2,0.8) {\ldots};
  \node[] (4) at (1,1.2) {$_{m}$};
 
   \draw (0) edge[->] node {} (1);   
   \draw (0) edge[->] node {} (2);   
   \draw (0) edge[->] node {} (4);

     \node[] (-1) at (-1,-1.2) {$_1$};
  \node[] (-2) at (-0.6,-1.2) {$_2$};
  \node[] (-.1) at (0.2,-.9) {\ldots};
  \node[] (-4) at (1,-1.2) {$_{n}$};

   \draw (0) edge[<-] node {} (-1);   
   \draw (0) edge[<-] node {} (-2);   
   \draw (0) edge[<-] node {} (-4);  
\end{tikzpicture} 
= \sum_{J_1\sqcup J_2 =[m]\atop
I_1\sqcup I_2=[n]} 
\begin{tikzpicture}[baseline=-0.55ex,scale=0.6]
  \node[int] (0) at (0,0) {};

  \node[] (1) at (-1,1.5) {};
  \node[] (2) at (-0.6,1.5) {};
  \node[] (.1) at (0.2,1.1) {$\ldots$};
  \node[] (4) at (1,1.5) {};
  
\node[] (.1) at (0,1.7) {$\overbrace{\qquad \quad }^{J_1}$};

  \node[int] (10) at (2,1) {};
    \node[] (11) at (1.1,-.6) {};
  \node[] (12) at (1.5,-0.6) {};
  \node[] (va) at (2.2,-0.2) {$\ldots$};
  \node[] (13) at (3,-0.6) {};  
  
  \node[] (.1) at (2,-0.9) {$\underbrace{\qquad \quad }_{ I_2}$};

  
   \draw (10) edge[<-] node {} (11);   
   \draw (10) edge[<-] node {} (12);
     \draw (10) edge[<-] node {} (13);
  
  \node[] (21) at (1.1,2.1) {};
  \node[] (22) at (1.5,2.1) {};
    \node[] (-.1) at (2.2,1.7) {$\ldots$};
  \node[] (23) at (3,2.1) {};

\node[] (.1) at (2,2.3) {$\overbrace{\qquad \quad }^{ J_2 }$};

     \draw (10) edge[->] node {} (21);   
   \draw (10) edge[->] node {} (22);
     \draw (10) edge[->] node {} (23);

    \draw (0) edge[->] node {} (10);

   \draw (0) edge[->] node {} (1);   
   \draw (0) edge[->] node {} (2);   
   \draw (0) edge[->] node {} (4);

     \node[] (-1) at (-1,-1) {};
  \node[] (-2) at (-0.6,-1) {};
  \node[] (-.1) at (0.2,-.7) {\ldots};
  \node[] (-4) at (1,-1) {};

\node[] (.1) at (0,-1.2) {$\underbrace{\qquad  \quad}_{I_1}$};  
   \draw (0) edge[<-] node {} (-1);   
   \draw (0) edge[<-] node {} (-2);   
   \draw (0) edge[<-] node {} (-4);  
\end{tikzpicture}.
$$
It was shown in \cite{PROPed} that $\hoLieB_{0,0}$ is a minimal resolution of $\LieB_{0,0},$ i.e.\ the natural projection map $p:\hoLieB_{0,0}\to \LieB_{0,0}$ is a quasi-isomorphism.

We define the \emph{ribbon graph complex} $(\RGC,\delta+\Delta_1)$ to be the deformation complex
$$(\RGC,\delta+\Delta_1) := \Def(\hoLieB_{0,0}\stackrel{s \circ p}{\rightarrow}  \RGra).$$
As a graded vector space
$$
\RGC\cong  \prod_{n,m\ge 1} \left(\RGra(n,m) \right)^{\sym_{n}\times\sym_m}
$$
is spanned by ribbon graphs in $\RGra$ coinvariant under the actions of permuting vertices and boundaries.

The differentials $\delta + \Delta_1$ act on a ribbon graph $\Gamma$ with $n$ vertices and $m$ boundaries by
$$\delta \begin{tikzpicture}[baseline=-0.55ex,scale=0.6]
  \node[] (0) at (0,0) {$_{\Gamma}$};
  \node[] (1) at (-1,1.2) {$_1$};
  \node[] (2) at (-0.6,1.2) {$_2$};
  \node[] (.1) at (0.2,0.8) {\ldots};
  \node[] (4) at (1,1.2) {$_{m}$};
 
   \draw (0) edge[->] node {} (1);   
   \draw (0) edge[->] node {} (2);   
   \draw (0) edge[->] node {} (4);

     \node[] (-1) at (-1,-1.2) {$_1$};
  \node[] (-2) at (-0.6,-1.2) {$_2$};
  \node[] (-.1) at (0.2,-.9) {\ldots};
  \node[] (-4) at (1,-1.2) {$_{n}$};

   \draw (0) edge[<-] node {} (-1);   
   \draw (0) edge[<-] node {} (-2);   
   \draw (0) edge[<-] node {} (-4);  
\end{tikzpicture} = \sum_{i=1}^{n} \begin{tikzpicture}[baseline=-0.55ex,scale=0.6]
  \node[] (0) at (0,0) {$_{\Gamma}$};
  \node[] (1) at (-1,1.2) {$_1$};
  \node[] (2) at (-0.6,1.2) {$_2$};
  \node[] (.1) at (0.2,0.8) {\ldots};
  \node[] (4) at (1,1.2) {$_{m}$};
 
   \draw (0) edge[->] node {} (1);   
   \draw (0) edge[->] node {} (2);   
   \draw (0) edge[->] node {} (4);

     \node[] (-1) at (-1,-1.2) {$_1$};
  \node[] (-2) at (-0.6,-1.2) {$_2$};
  \node[] (-.1) at (-0.2,-.9) {\tiny{\ldots}};
  \node[] (-.1) at (0.5,-.9) {\tiny{\ldots}};
  \node[] (-4) at (1,-1.2) {$_{n}$};

   \draw (0) edge[<-] node {} (-1);   
   \draw (0) edge[<-] node {} (-2);   
   \draw (0) edge[<-] node {} (-4);  

  \node[int] (i) at (0.2,-2) {};

  \node[] (i') at (-0.3,-3) {$_i$};

  \node[] (i'1) at (0.7,-3) {$_{n+1}$};
   \draw (0) edge[<-] node[near end,left] {$_i$} (i);  
 \draw (i) edge[<-] (i');  
\draw (i) edge[<-] (i'1);  
\end{tikzpicture}-\sum_{i=1}^{m} \begin{tikzpicture}[baseline=-0.55ex,scale=0.6]
  \node[] (0) at (0,0) {$_{\Gamma}$};
  \node[] (1) at (-1,1.2) {$_1$};
  \node[] (2) at (-0.6,1.2) {$_2$};
  \node[] (.1) at (0.2,0.8) {\ldots};
  \node[] (4) at (1,1.2) {$_{m}$};
 
   \draw (0) edge[->] node {} (1);   
   \draw (0) edge[->] node {} (2);   
   \draw (0) edge[->] node {} (4);

     \node[] (-1) at (-1,-1.2) {$_1$};
  \node[] (-2) at (-0.6,-1.2) {$_2$};
  \node[] (-.1) at (-0.2,-.9) {\tiny{\ldots}};
  \node[] (-.1) at (0.5,-.9) {\tiny{\ldots}};
  \node[] (-4) at (1,-1.2) {$_{n}$};

   \draw (0) edge[<-] node {} (-1);   
   \draw (0) edge[<-] node {} (-2);   
   \draw (0) edge[<-] node {} (-4);  

  \node[int] (i) at (0.2,2) {};

  \node[] (i') at (-0.3,3) {$_i$};

  \node[] (i'1) at (1.7,1.5) {$_{n+1}$};
   \draw (0) edge[->] node[near end,left] {$_i$} (i);  
 \draw (i) edge[->] (i');  
\draw (i) edge[<-] (i'1);  
\end{tikzpicture},$$
$$\Delta_1\begin{tikzpicture}[baseline=-0.55ex,scale=0.6]
  \node[] (0) at (0,0) {$_{\Gamma}$};
  \node[] (1) at (-1,1.2) {$_1$};
  \node[] (2) at (-0.6,1.2) {$_2$};
  \node[] (.1) at (0.2,0.8) {\ldots};
  \node[] (4) at (1,1.2) {$_{m}$};
 
   \draw (0) edge[->] node {} (1);   
   \draw (0) edge[->] node {} (2);   
   \draw (0) edge[->] node {} (4);

     \node[] (-1) at (-1,-1.2) {$_1$};
  \node[] (-2) at (-0.6,-1.2) {$_2$};
  \node[] (-.1) at (0.2,-.9) {\ldots};
  \node[] (-4) at (1,-1.2) {$_{n}$};

   \draw (0) edge[<-] node {} (-1);   
   \draw (0) edge[<-] node {} (-2);   
   \draw (0) edge[<-] node {} (-4);  
\end{tikzpicture} = \sum_{i=1}^{m} \begin{tikzpicture}[baseline=-0.55ex,scale=0.6]
  \node[] (0) at (0,0) {$_{\Gamma}$};
  \node[] (1) at (-1,1.2) {$_1$};
  \node[] (2) at (-0.6,1.2) {$_2$};
  \node[] (.1) at (0.2,0.8) {\ldots};
  \node[] (4) at (1,1.2) {$_{m}$};
 
   \draw (0) edge[->] node {} (1);   
   \draw (0) edge[->] node {} (2);   
   \draw (0) edge[->] node {} (4);

     \node[] (-1) at (-1,-1.2) {$_1$};
  \node[] (-2) at (-0.6,-1.2) {$_2$};
  \node[] (-.1) at (-0.2,-.9) {\tiny{\ldots}};
  \node[] (-.1) at (0.5,-.9) {\tiny{\ldots}};
  \node[] (-4) at (1,-1.2) {$_{n}$};

   \draw (0) edge[<-] node {} (-1);   
   \draw (0) edge[<-] node {} (-2);   
   \draw (0) edge[<-] node {} (-4);  

  \node[int] (i) at (0.2,2) {};

  \node[] (i') at (-0.3,3) {$_i$};

  \node[] (i'1) at (0.9,3) {$_{m+1}$};
   \draw (0) edge[->] node[near end,left] {$_i$} (i);  
 \draw (i) edge[->] (i');  
\draw (i) edge[->] (i'1);  
\end{tikzpicture}-
\
\sum_{i=1}^{n} \begin{tikzpicture}[baseline=-0.55ex,scale=0.6]
  \node[] (0) at (0,0) {$_{\Gamma}$};
  \node[] (1) at (-1,1.2) {$_1$};
  \node[] (2) at (-0.6,1.2) {$_2$};
  \node[] (.1) at (0.2,0.8) {\ldots};
  \node[] (4) at (1,1.2) {$_{m}$};
 
   \draw (0) edge[->] node {} (1);   
   \draw (0) edge[->] node {} (2);   
   \draw (0) edge[->] node {} (4);

     \node[] (-1) at (-1,-1.2) {$_1$};
  \node[] (-2) at (-0.6,-1.2) {$_2$};
  \node[] (-.1) at (-0.2,-.9) {\tiny{\ldots}};
  \node[] (-.1) at (0.5,-.9) {\tiny{\ldots}};
  \node[] (-4) at (1,-1.2) {$_{n}$};

   \draw (0) edge[<-] node {} (-1);   
   \draw (0) edge[<-] node {} (-2);   
   \draw (0) edge[<-] node {} (-4);  

  \node[int] (i) at (0.2,-2) {};

  \node[] (i') at (-0.3,-3) {$_i$};

  \node[] (i'1) at (1.7,-1.5) {$_{m+1}$};
   \draw (0) edge[<-] node[near end,left] {$_i$} (i);  
 \draw (i) edge[<-] (i');  
\draw (i) edge[->] (i'1);  
\end{tikzpicture},
$$
where
$$
\begin{tikzpicture}[baseline=-.55ex,scale=0.6]
 \node[int] (a) at (0,0) {};
  \node[](out1) at (1,-1) {$_{(m+1)}$};
   \node[](out2) at (-1,-1) {$_i$};
   \node[](in) at (0,1) {$_i$};
     \draw (out1) edge[->] (a);
      \draw (out2) edge[->] (a);
       \draw (a) edge[->] (in);
\end{tikzpicture}=  \begin{tikzpicture}[baseline=-.55ex,scale=0.6]
 \node[draw,circle] (a) at (0,0) {$_i$};
 \node[draw,circle] (b) at (1.8,0) {$_{m+1}$};
 \node[] (c) at (0.75,0.5) {$_i$}; 
 \draw (a) edge[->] (b);
\end{tikzpicture},
\quad \text{ and }
 \begin{tikzpicture}[baseline=-.55ex,scale=0.6]
 \node[int] (a) at (0,0) {};
  \node[](out1) at (1,1) {$_{m+1}$};
   \node[](out2) at (-1,1) {$_i$};
   \node[](in) at (0,-1) {$_i$};
     \draw (a) edge[->] (out1);
      \draw (a) edge[->] (out2);
       \draw (in) edge[->] (a);
\end{tikzpicture}=\begin{tikzpicture}[baseline=-.55ex,scale=0.6]
 \node[draw,circle] (A) at (0,0) {$_i$};
  \draw (A) edge[loop]  node[midway, below]{$_{i}$} node[midway, above]{$_{m+1}$}(a); 	
\end{tikzpicture}.
$$
In words, the first part of the differential, $\delta$, splits vertices in a way that respects the cyclic ordering. The other part of the differential, $\Delta_1$, adds an edge between each pair of corners of each boundary.

\subsection{The Map $F:\mO \GC_1 \to \RGC$ }
Consider the deformation complex
$$
\Def(\hoLieB_{0,0} \to \LieB_{0,0}) \cong
\prod_{n,m\ge 1}\left(\LieB_{0,0}(n,m)\right)^{{\sym_{n}\times\sym_m}}.
$$
It is spanned by oriented graphs with all vertices 3-valent, with ingoing and outgoing hairs, without internal sources or targets, modulo the relations \eqref{eq:jacobi},\eqref{eq:cojac},\eqref{eq:IHXor}.
In \cite{MW2}, S. Merkulov and T. Willwacher constructed a quasi-isomorphism
\begin{eqnarray*}
F_1:\mO \GC_1 &\to & \Def(\hoLieB_{0,0} \to \LieB_{0,0})[1]\\
\Gamma &\mapsto & F_1(\Gamma).
\end{eqnarray*}
The element $F_1(\Gamma)$ is obtained from a graph $\Gamma\in \mO \GC_1$ by attaching an incoming leg to each source, an outgoing leg to each target, and setting it to $0$ if it contains a vertex that is not $3$-valent.

Next, the map $s:\LieB_{0,0}\to \RGra$ from Proposition \ref{prop:LieB_{0,0}-RGra} gives us a map of complexes
$$F_2:\Def(\hoLieB_{0,0}\to \LieB_{0,0})[1]  \to (\RGC[1],\delta+\Delta_1)$$
by $F_2(\Gamma: \hoLieB_{0,0}\to \LieB_{0,0}) := s\circ \Gamma: \hoLieB_{0,0}\to \RGra$.
We now have our map
\begin{equation} \label{eq:F}
F:=F_2\circ F_1:\mO \GC_1 \to  \Def(\hoLieB_{0,0}\to \LieB_{0,0})[1]  \to (\RGC[1],\delta+\Delta_1).
\end{equation}

We note that $F_2$ maps a graph in $\LieB_{0,0}(n,m)^{\sym_{n}\times\sym_m}$ to a sum of ribbon graphs with $n$ vertices and $m$ boundaries. We also have that $F_1$ maps a graph $\Gamma\in \mO \GC_1$ with $n$ sources and $m$ targets to a graph in $\LieB_{0,0}(n,m)^{\sym_{n}\times\sym_m}$.

Hence $F$ maps a graph $\Gamma$ with $m$ target vertices to a sum of ribbon graphs with $m$ boundaries. We may take a filtration on $\RGC$ by the number of boundaries, and a filtration on $\mO\GC_1$ by the number of target vertices. Then
$$gr(\RGC,\delta+\Delta_1)= (\RGC,\delta),$$
and
$$
gr(\mO \GC_1, \delta) = (\mO \GC_1, \delta_0).
$$
As $F$ maps a graph with $m$ target vertices to a sum of ribbon graphs with precisely $m$ boundaries, we get that
$$
gr F:gr(\mO \GC_1,\delta)\to  gr(\RGC[1],\delta+\Delta_1)
$$
is given by the same map of vector spaces
$$F:(\mO \GC_1,\delta_0)\to (\RGC[1],\delta).$$
We have now established both maps mentioned in Corollary \ref{cor:Mgn}.
\begin{thm} [Corollary \ref{cor:Mgn}]
We have a zig-zag of morphisms
$$(\HGC_{0}, \delta)\gets (\mO \GC_1, \delta_0)\to (\RGC[1],\delta),$$
where the left map is a quasi-isomorphism, given explicitly in \eqref{eq:G} and \eqref{eq:Ge}, and the right map is given in \eqref{eq:F}.
\end{thm}


\begin{thebibliography}{10}


\bibitem{AT}
Arone, G.; Turchin, V.
{\it Graph-complexes computing the rational homotopy of high dimensional analogues of spaces of long knots}.
Ann. Inst. Fourier (Grenoble) 65 (2015), no. 1, 1–62.

\bibitem{CGP1}
Chan, M.; Galatius, S.; Payne, S.
{\it Tropical curves, graph homology, and top weight cohomology of M\_g}.
arXiv:1805.10186

\bibitem{CGP2}
Chan, M.; Galatius, S.; Payne, S.
{\it Topology of moduli spaces of tropical curves with marked points}.
arXiv:1903.07187

\bibitem{FTW}
Fresse, B.; Turchin, V.; Willwacher, T.
{\it The rational homotopy of mapping spaces of $E_n$ operads}.
arXiv:1703.06123 

\bibitem{DGC1}
Khoroshkin, A.; Willwacher, T.; \v Zivkovi\' c, M.
{\it Differentials on graph complexes}.
Adv. Math. 307 (2017), 1184–1214.

\bibitem{DGC2}
Khoroshkin, A.; Willwacher, T.; \v Zivkovi\' c, M.
{\it Differentials on graph complexes II: hairy graphs}.
Lett. Math. Phys. 107 (2017), no. 10, 1781–1797.

\bibitem{Kont1}
Kontsevich, M.
{\it Formal (non)commutative symplectic geometry}.
The Gelfand Mathematical Seminars, 1990–1992, 173–187, Birkh\" auser Boston, Boston, MA, 1993.

\bibitem{Kont2}
Kontsevich, M.
{\it Formality conjecture}.
Deformation theory and symplectic geometry (Ascona, 1996), 139–156, Math. Phys. Stud., 20, Kluwer Acad. Publ., Dordrecht, 1997. 


\bibitem{Kont3}
Kontsevich, M.
{\it Intersection  Theory  on  the  Moduli  Space  of  Curves and  the  Matrix  Airy  Function}. Commun.Math. Phys. 147, 1–23 (1992) doi:10.1007/BF02099526.

\bibitem{PROPed}
Markl, M.; Voronov, A. A.
{\it PROPped-up graph cohomology}.
Algebra, arithmetic, and geometry: in honor of Yu. I. Manin. Vol. II, 249–281, Progr. Math., 270, Birkhäuser Boston, Inc., Boston, MA, 2009.

\bibitem{Merk1}
Merkulov, S.
{\it Multi-oriented props and homotopy algebras with branes}.
arXiv:1712.09268.



\bibitem{MW}
Merkulov, S.; Willwacher, T.
{\it Props of ribbon graphs, involutive Lie bialgebras and moduli spaces of curves}.
arXiv:1511.07808.


\bibitem{MW2}
Merkulov, S.; Willwacher, T.
{\it Deformation theory of Lie bialgebra properads}.
Geometry and physics. Vol. I, 219–247, Oxford Univ. Press, Oxford, 2018.


\bibitem{Penner}
Penner, R. C.
{\it Perturbative series and the moduli space of Riemann surfaces}.
J. Differential Geom. 27 (1988), no. 1, 35–53.

\bibitem{Tur1}
Turchin, V.
{\it Hodge-type decomposition in the homology of long knots}.
J. Topol. 3 (2010), no. 3, 487–534.

\bibitem{TW1}
Turchin, V.; Willwacher, T.
{\it Relative (non-)formality of the little cubes operads and the algebraic Cerf lemma}.
Amer. J. Math. 140 (2018), no. 2, 277–316.

\bibitem{TW2}
Turchin, V.; Willwacher, T.
{\it Commutative hairy graphs and representations of Out(Fr)}.
J. Topol. 10 (2017), no. 2, 386–411.

\bibitem{Will}
Willwacher, T.
{\it Deformation quantization and the Gerstenhaber structure on the homology of knot spaces}.
arXiv:1506.07078.

\bibitem{grt}
Willwacher, T.
{\it M. Kontsevich's graph complex and the Grothendieck-Teichmüller Lie algebra}.
Invent. Math. 200 (2015), no. 3, 671–760.

\bibitem{Poly}
Willwacher, T.
{\it Stable cohomology of polyvector fields}.
Math. Res. Lett. 21 (2014), no. 6, 1501–1530.

\bibitem{oriented}
Willwacher, T.
{\it The oriented graph complexes}.
Comm. Math. Phys. 334 (2015), no. 3, 1649–1666.

\bibitem{eulerchar}
Willwacher, T.; \v Zivkovi\' c, M.
{\it Multiple edges in M. Kontsevich's graph complexes and computations of the dimensions and Euler characteristics}.
Adv. Math. 272 (2015), 553–578.

\bibitem{DGC3}
\v Zivkovi\' c, M.
{\it Differentials on graph complexes III: hairy graphs and deleting a vertex}.
Lett. Math. Phys. 109 (2019), no. 4, 975–1054.

\bibitem{Multi}
\v Zivkovi\' c, M.
{\it Multi-directed graph complexes and quasi-isomorphisms between them I: oriented graphs}.
\newblock High. Struct. 4(1): 266–283, 2020.

\bibitem{MultiSourced}
\v Zivkovi\' c, M.
{\it Multi-directed Graph Complexes and Quasi-isomorphisms Between Them II: Sourced Graphs}.
Int. Math. Res. Not. IMRN (2019), rnz212.
\end{thebibliography}
\end{document}